\documentclass[11pt]{amsart} 
\usepackage{verbatim, latexsym, amssymb, amsmath,color}
\usepackage{epsfig} 

\numberwithin{equation}{section}

\def\R{\mathbb R}

\def\N{\mathbb N}

\def\H{\mathcal H}

\newtheorem{thm}{Theorem}[section]

\theoremstyle{remark}

\theoremstyle{definition}

\title{Morse inequalities for the area functional}
\author{Fernando C. Marques, Rafael Montezuma and  Andr\'e Neves}
\address{Princeton University \\ Princeton NJ 08544 \\USA}
\email{coda@math.princeton.edu}
\address{University of Massachusetts \\ Department of Mathematics \\ Amherst MA  01003\\ USA / Universidade Federal do Cear\'{a}\\ Departamento de Matem\'{a}tica \\  Fortaleza, CE 60455-760 \\ Brazil}
\email{rmontezuma@math.umass.edu}
\address{University of Chicago \\ Department of Mathematics \\ Chicago IL 60637\\ USA}
\email{aneves@uchicago.edu}
\thanks{The first author is partly supported by NSF-DMS-1811840. The third author is partly supported by NSF  DMS-1710846 and a Simons Investigator Grant.}

\begin{document}

\maketitle

\begin{abstract}
In this article we prove the strong Morse inequalities  for the area functional in codimension one, assuming that the ambient dimension  satisfies $3\leq (n+1)\leq 7$, in both  the closed and the boundary cases.
 \end{abstract}

\section{Introduction}

Morse theory relates the structure of the set of critical points of a real-valued function to the topology of the space on which the function is defined. It was initially devised to study  geodesics. In this paper we establish a Morse theory for minimal hypersurfaces in a   Riemannian manifold $(M^{n+1},g)$, assuming $3 \leq (n+1) \leq 7$, by proving that the strong Morse inequalities hold 
(\cite{morse-book}).

The area functional  is defined on the space $\mathcal{Z}_n(M^{n+1}, \mathbb{Z}_2)$ of $n$-dimensional mod 2 flat cycles $T=\partial U$, endowed with the flat topology.
Denote by $b_k(a)$ the $k$-th  Betti number   of the space of cycles with area less than $a$ (with coefficients in $\mathbb{Z}_2$) and by $c_k(a)$ the number of closed, embedded, minimal hypersurfaces in $\mathcal{Z}_n(M^{n+1}, \mathbb{Z}_2)$ with area less than $a$ and index equal to $k$ (notice that the trivial cycle counts).

\subsection{Theorem}\label{theorem.closed}
 {\em  Suppose $g$ is a $C^\infty$-generic (bumpy) Riemannian metric. For each $a\in (0,\infty)$, we have $b_k(a)<\infty$ for every $k\in \mathbb{Z}_+$ and the Strong Morse Inequalities hold:
 \begin{eqnarray*}
 c_r(a)-c_{r-1}(a)+\cdots + (-1)^rc_0(a) \geq b_r(a)-b_{r-1}(a)+\cdots + (-1)^rb_0(a)
 \end{eqnarray*}
 for every $r \in \mathbb{Z}_+$. In particular,
 \begin{eqnarray*}
 c_r(a)\geq b_r(a)
 \end{eqnarray*}
  for every $r \in \mathbb{Z}_+$.
  }
  
    
  In the case of  two-dimensional parametrized minimal surfaces spanning a wire in Euclidean space, the  inequalities were proven by Morse and Tompkins \cite{morse-tompkins}. Such a theory is also in Shiffman \cite{shiffman}. Struwe \cite{struwe} found a clear proof from the point of view of functional analysis, using Sobolev spaces and  Palais-Smale  type compactness (\cite{palais-smale}).
  
Several techniques have been developed recently  to characterize the Morse index of  minimal hypersurfaces produced by min-max methods (\cite{marques-neves-index}, \cite{marques-neves-lower-bound}). Almgren \cite{almgren-varifolds} and Pitts \cite{pitts} had shown how to construct closed minimal hypersurfaces by doing min-max over one-parameter sweepouts, but no information on the Morse index was given. The papers \cite{marques-neves-index} and \cite{marques-neves-lower-bound} use the fact that $\mathcal{Z}_n(M^{n+1}, \mathbb{Z}_2)$ has the same cohomology groups of $\mathbb{RP}^\infty$, which implies the existence of multiparameter sweepouts.  In fact,  $H^k(\mathcal{Z}_n(M^{n+1}, \mathbb{Z}_2), \mathbb{Z}_2)=\mathbb{Z}_2$ for every $k \in \mathbb{N}$. 
  
  The point
  was to show that generically, in which case the metric is bumpy meaning  closed minimal hypersurfaces are nondegenerate (White \cite{white2}, \cite{white3}), a  closed minimal hypersurface produced by min-max over $k$-parameter sweepouts must have index equal to $k$. General upper bounds were obtained in \cite{marques-neves-index} and lower bounds were proven in \cite{marques-neves-lower-bound} assuming the Multiplicity One Conjecture: any component of a min-max minimal hypersurface has multiplicity one.  As in \cite{marques-neves-lower-bound}, our paper uses the local min-max property of White \cite{white-minmax}.
  
 Xin Zhou proved the Multiplicity One Conjecture in  \cite{zhou-multiplicity} for the Almgren-Pitts setting, using the Zhou-Zhu  \cite{zhou-zhu}      
 prescribed mean curvature version of the theory. If  $(n+1)=3$,  for the Allen-Cahn variant of the theory,  the conjecture had been proven by Chodosh and Mantoulidis \cite{chodosh-mantoulidis}. Combining the results of \cite{marques-neves-index}, \cite{marques-neves-lower-bound} and \cite{zhou-multiplicity}, for a generic metric $g$,   there exists a closed minimal hypersurface $\Sigma_k\subset (M^{n+1},g)$ with ${\rm index}(\Sigma_k)=k$ for every $k\in \mathbb{N}$.
 Since $b_k(\mathcal{Z}_n(M^{n+1}, \mathbb{Z}_2))=1$, this corresponds to the inequality $c_k\geq b_k$. Notice that the area of $\Sigma_k$ is equal to the $k$-width $\omega_k(M,g)$ of $M$ (see \cite{marques-neves-lower-bound} for the definition).  For more details see the survey paper \cite{marques-neves-morse-survey}.

 Theorem \ref{theorem.closed}  generalizes this by accounting  for all (separating) closed minimal hypersurfaces in $(M,g)$, and not just those realizing the volume spectrum. 
  
 We also consider the boundary case in this paper. De Lellis and Ramic \cite{delellis-ramic} proved a mountain pass theorem using sweepouts based on those of \cite{delellis-tasnady}. This establishes that if there are two strictly stable, embedded minimal hypersurfaces with the same boundary, and assuming there are no closed minimal hypersurfaces, then there must be a third minimal hypersurface with that boundary. The boundary is assumed to be contained in a convex hypersurface of the ambient space, in which case regularity holds. Due to the properties of the sweepouts used in \cite{delellis-ramic}, they also need the strictly stable hypersurfaces to bound an open set.
 
A mountain pass theorem for the Almgren-Pitts setting was achieved in \cite{montezuma}. The boundary is still assumed to be in a convex hypersurface, but the condition of the hypersurfaces bounding an open set is not imposed. The strictly stable hypersurfaces are assumed to be homologous.
This allows the construction of a path between them that is continuous in the flat topology. The fact that the boundary is nonempty implies the mutiplicity one property,  which is crucial for the arguments of our paper.

For the boundary version of the  Morse inequalities, we assume $(M^{n+1},g)$, $3\leq (n+1) \leq 7$, has strictly convex boundary $\partial M$. We suppose  $M$
contains no closed minimal hypersurfaces. Let $\gamma^{n-1}\subset \partial M$ be an $(n-1)$-dimensional smooth, closed submanifold.
We assume every minimal hypersurface $\Sigma \subset M$ with $\partial\Sigma=\gamma$ is nondegenerate, which holds true for Baire generic boundaries $\gamma\subset \partial M$ (White \cite{white-parametric}).

\subsection{Theorem}\label{morse.inequalities.1}{\em The Strong Morse Inequalities for the area functional hold, i.e.
 $$
 c_k(\gamma)-c_{k-1}(\gamma)+ \cdots +(-1)^kc_0(\gamma) \geq   (-1)^k
 $$
 for every $k\geq 0$. Here $c_k(\gamma)$ denotes the number of minimal hypersurfaces of index $k$ with boundary $\gamma$.}
 \medskip
 
 The mountain pass theorem when $3 \leq (n+1) \leq 7$ is a consequence of the strong Morse inequalities, since we have
 $$
 c_1(\gamma) \geq c_0(\gamma)-1.
 $$

 There is a version of Theorem \ref{morse.inequalities.1} counting only hypersurfaces with area less than $a$, like in Theorem \ref{theorem.closed}, which can be proven with similar methods.
 If there are closed minimal hypersurfaces in $M$, they will be achieved with multiplicity one if the metric is also bumpy. The strong Morse 
 inequalities will continue to hold in that case for hypersurfaces with area less than $a$. 
 
 If $(n+1)\geq 8$,  Schoen-Simon regularity theory (\cite{schoen-simon}) for stable minimal hypersurfaces implies that min-max minimal hypersurfaces are smooth outside a possible singular set of codimension 7. The Morse index of the minimal variety is the index of the smooth part (Dey \cite{dey}). Upper bounds for the Morse index were proven by Li \cite{li}. A multiplicity one property and lower bounds are still open. The mountain pass theorems of \cite{delellis-ramic} and \cite{montezuma} hold for $(n+1) \geq 8$ with singular hypersurfaces. The third solution should have index one.
 
The paper is organized as follows. In Section 2, we introduce some notation of Geometric Measure Theory. In Section 3, we prove the
strong Morse inequalities in the closed case. In Section 4, we prove a compactness theorem for the boundary case. In Section 5, we prove
the strong Morse inequalities for the boundary case.

\medskip

{\it Acknowledgments:} The authors would like to thank   the support of the  Institute for Advanced Study, where part of this
work was conducted.

\section{Preliminaries}

Let $(M^{n+1},g)$ be an $(n+1)$-dimensional closed Riemannian manifold.  We assume, for convenience, that $(M,g)$ is isometrically embedded in some Euclidean space $\mathbb{R}^L$.

 The spaces we will work with in this paper are:
\begin{itemize}
\item the space ${\bf I}_{l}(M;\mathbb{Z}_2)$  of $l$-dimensional  flat chains    in $\mathbb{R}^L$ with coefficients in $\mathbb{Z}_2$ and support contained  in $M$, where $l=n$ or  $n+1$; 
\item the space ${\mathcal Z}_n(M;\mathbb{Z}_2)$  of flat chains  $T \in {\bf I}_n(M;\mathbb{Z}_2)$ such that there exists $U \in {\bf I}_{n+1}
(M;\mathbb{Z}_2)$ with  $\partial U=T$;
\item the closure $\mathcal{V}_n(M)$, in the weak topology, of the space of $n$-dimensional rectifiable varifolds in $\mathbb{R}^L$ with support contained in $M$. 
\end{itemize}
We assume implicitly that ${\bf M}(T)+{\bf M}(\partial T)<\infty$ for every $T\in {\bf I}_{l}(M;\mathbb{Z}_2)$, where ${\bf M}$ denotes the mass functional. We will refer to $\tilde{{\mathcal Z}}_n(M;\mathbb{Z}_2)$ as the {\it space of cycles}. Flat chains over a finite coefficient group were introduced by Fleming \cite{fleming}.

Given $T\in {\bf I}_l(M;\mathbb{Z}_2)$,  we denote by $|T|$ and $||T||$ the integral varifold   and the Radon measure in $M$ associated with $|T|$, respectively;  given $V\in \mathcal{V}_n(M)$, $||V||$ denotes the Radon measure in $M$ associated with $V$.  The space of $n$-dimensional integral 
varifolds with support in $M$ is denoted by $\mathcal{IV}_n(M)$.

 The  spaces above come with several relevant metrics. For instance, the metric ${\bf M}(T_1,T_2)={\bf M}(T_1-T_2)$ defines the mass topology. The  {\it flat metric} 
 $$
 \mathcal F(T_1,T_2) = \inf \{{\bf M}(Q)+{\bf M}(R): T_1-T_2=Q+\partial R\}
 $$
 induces the flat topology (we put 
$\mathcal{F}(T)=\mathcal{F}(T,0)$). The  ${\bf F}$-{\it metric}  is defined in  the book of Pitts  \cite[page 66]{pitts} and   induces the varifold weak topology on $\mathcal{V}_n(M)\cap \{V: ||V||(M) \leq a\}$ for any $a$. It satisfies 
$$||V||(M) \leq ||W||(M) +{\bf F}(V,W)$$ for all $V,W \in \mathcal{V}_n(M)$. We denote by ${\overline{\bf B}^{\bf F}_{\delta}(V)}$ and ${\bf B}^{\bf F}_{\delta}(V)$ the closed and open metric balls, respectively, with radius $\delta$ and center $V \in \mathcal{V}_n(M)$. Similarly, we denote by ${\overline{\bf B}^{\mathcal F}_{\delta}(T)}$ and ${\bf B}^{\mathcal F}_{\delta}(T)$ the corresponding balls with  center $T \in \mathcal{Z}_n(M;\mathbb{Z}_2)$ in the flat metric. 
Finally,  the ${\bf F}$-{\it metric} on ${\bf I}_l(M;\mathbb{Z}_2)$ is defined by
$$ {\bf F}(S,T)=\mathcal{F}(S-T)+{\bf F}(|S|,|T|).$$
We have ${\bf F}(|S|,|T|) \leq {\bf M}(S,T)$ and hence $ {\bf F}(S,T) \leq 2{\bf M}(S,T)$ for any $S,T \in {\bf I}_l(M;\mathbb{Z}_2)$.

We assume that  ${\bf I}_l(M;\mathbb{Z}_2)$ and  $\tilde{{\mathcal Z}}_n(M;\mathbb{Z}_2)$ have the topology induced by the flat metric. When endowed with
the topology of the ${\bf F}$-metric or the mass norm, these spaces will be denoted by  ${\bf I}_l(M;{\bf F};\mathbb{Z}_2)$, $\tilde{{\mathcal Z}}_n(M;{\bf F};\mathbb{Z}_2)$, ${\bf I}_l(M;{\bf M};\mathbb{Z}_2)$, $\tilde{{\mathcal Z}}_n(M;{\bf M};\mathbb{Z}_2)$, respectively.

\section{Morse inequalities}

Let $\mathcal{M}$ be the space of all smooth Riemannian metrics on $M$. Given $g\in \mathcal{M}$, let $\tilde{\mathfrak{M}}_g$ denote the collection of all smooth, closed, embedded,  $g$-minimal hypersurfaces in $M$ and let $\mathfrak{M}_g$ be the subset of $\tilde{\mathfrak{M}}_g$ consisting of connected minimal hypersurfaces.

Suppose that the Riemannian metric $g$ on the closed manifold $M^{n+1}$,  $3\leq (n+1)\leq 7$, is bumpy.
 This means that every closed minimal hypersurface $\Sigma \subset M$  is  a nondegenerate critical point of the area functional, including immersed hypersurfaces. This condition holds for $C^\infty$ generic metrics (\cite{white2}, \cite{white3}, \cite{AmbCarSha}).
 
 Given $k\in \mathbb{Z}_+$ and $a \in (0,\infty)$, let $c_k(a)$ denote the number of smooth, embedded, closed, minimal hypersurfaces $\Sigma \subset M$ in $\mathcal{Z}_n(M^{n+1},\mathbb{Z}_2)$ with ${\rm index}(\Sigma)=k$ and ${\rm area}(\Sigma)<a$. The hypersurface $\Sigma$ could be disconnected, in which case the index is the sum of the indices of the components. If we want to emphasize the dependence on the Riemannian metric $g$, we will write $c_{k,g}(a)$ instead of $c_k(a)$.

 We denote by $b_k(a)$ the  $k$-th Betti number of the topological space
 $$
\mathcal{Z}^a = \{ T\in \mathcal{Z}_n(M,{\bf F};\mathbb{Z}_2): {\bf M}(T)<a\},
 $$
 endowed with the ${\bf F}$-metric, with coefficients in $\mathbb{Z}_2$. (The  inequalities also hold for the Betti numbers with coefficients in any field, by adapting the methods.) The set $\mathcal{Z}^a$ is open in $\mathcal{Z}_n(M,{\bf F};\mathbb{Z}_2)$. If we want to specify the dependence
 on the metric $g$, we will write $\mathcal{Z}^a_g$.

 \subsection{Theorem}\label{theorem.closed.2}
 {\em  For each $a\in (0,\infty)$, we have $b_k(a)<\infty$ for every $k\in \mathbb{Z}_+$ and the Strong Morse Inequalities hold:
 \begin{eqnarray*}
 c_r(a)-c_{r-1}(a)+\cdots + (-1)^rc_0(a) \geq b_r(a)-b_{r-1}(a)+\cdots + (-1)^rb_0(a)
 \end{eqnarray*}
 for every $r \in \mathbb{Z}_+$. In particular,
 \begin{eqnarray*}
 c_r(a)\geq b_r(a)
 \end{eqnarray*}
  for every $r \in \mathbb{Z}_+$.
  }
  \medskip
  
  Notice that:
   \subsection{Lemma}
 {\em We have $c_k(a)<\infty$ for every $k$ and $a$.}

\begin{proof}
This follows immediately from Sharp's Compactness Theorem (\cite{sharp}), since the metric is bumpy.
\end{proof}

  We will use:
  \subsection{Proposition}\label{rational.independence.closed} {\em For a $C^\infty$-generic Riemannian metric $g$ on $M$, we have:
\begin{itemize}
\item every $g$-minimal hypersurface is $g$-nondegenerate;
\item and if 
$$
p_1 \cdot {\rm area}_{g}(\Sigma_1) + \cdots + p_N \cdot {\rm area}_{g}(\Sigma_N)=0,
$$
with $\{p_1, \dots, p_N\} \subset \mathbb{Z}$, $\{\Sigma_1, \dots, \Sigma_N\}\subset \mathfrak{M}_{g},$ and $\Sigma_k \neq \Sigma_l$ whenever $k\neq l$, then
$$
p_1 = \cdots = p_N = 0.
$$
\end{itemize}
 }

 \begin{proof}

Let $\mathcal{U}_{p,\alpha}$ be the set of Riemannian metrics $g\in \mathcal{M}$ such that:
\begin{itemize}
\item every $g$-minimal hypersurface with  index at most $p$ and area at most $\alpha$ is $g$-nondegenerate;
\item and if 
$$
m_1 \cdot {\rm area}_{g}(\Sigma_1) + \cdots + m_N \cdot {\rm area}_{g}(\Sigma_N)=0,
$$
with $\{m_1, \dots, m_N\} \subset \mathbb{Z}$,  $\{\Sigma_1, \dots, \Sigma_N\}\subset \mathfrak{M}_{g},$ $|m_k|\leq p$, $\text{index}(\Sigma_k) \leq p$ and $\text{area}(\Sigma_k)\leq \alpha$ for every $k$, and $\Sigma_k \neq \Sigma_l$ whenever $k\neq l$, then
$$
m_1 = \cdots = m_N = 0.
$$
\end{itemize}

\subsection{Claim} {\em The set $\mathcal{U}_{p,\alpha}$ is open and dense in $\mathcal{M}$, for every $p\in \mathbb{Z}_+$ and $\alpha>0$.}

\medskip

Openness follows  from Sharp's Compactness Theorem \cite{sharp}.

Let $g \in \mathcal{M}$ and let $\mathcal{V}\subset  \mathcal{M}$ be a $C^\infty$-neighborhood of $g$. By White (\cite{white2}, \cite{white3}), there exists a bumpy metric $g'\in \mathcal{V}$.  

By Sharp's Compactness Theorem (\cite{sharp}), there are only finitely many connected $g'$-minimal hypersurfaces  with index at most $p$ and area at most $\alpha$. Let $\{S_1, \dots, S_q\}$ be the collection of such hypersurfaces, with $S_k\neq S_l$ whenever $k\neq l$.

Recall that if $\tilde{g}=\exp(2\phi)g'$, then the second fundamental form of $\Sigma$ with respect to $\tilde{g}$ is given by (Besse \cite{besse}, Section 1.163)
\begin{eqnarray*}\label{second.fundamental.form}
A_{\Sigma, \tilde{g}} =  A_{\Sigma,g'} -  g' \cdot (\nabla \phi)^\perp,
\end{eqnarray*}
where $(\nabla \phi)^\perp(x)$ is the component of $\nabla \phi$ normal to $T_x\Sigma$.

We can pick $p_l \in S_l \setminus (\cup_{k\neq l} S_k)$ for every $l=1, \dots, q$ (see the proof of Lemma 4 of 
\cite{marques-neves-song}). Let $\varepsilon>0$ be sufficiently small so that $B_\varepsilon(p_k)\cap B_\varepsilon(p_l) = \emptyset$ whenever $k\neq l$, $B_\varepsilon(p_l) \cap (\cup_{k\neq l} S_k) = \emptyset$ for every $l=1, \dots, q$. We choose a nonnegative function $f_l\in C_c^\infty(B_\varepsilon(p_l))$, ${f_l}_{|S_l}\not \equiv 0$, such that $(\nabla_{g'} f_l)(x) \in T_xS_l$ for every $x\in S_l$. Hence $S_l$ is still minimal with respect to the metric $\hat{g}(t_1,\dots,t_q) = \exp(2(t_1f_1+\cdots+t_qf_q))g'$, for every $l=1, \dots, q$. 

Let $(t_1^{(i)}, \dots, t_q^{(i)})\in (0,1]^q$ be a sequence converging to zero so that, by putting $g_i=\hat{g}(t_1^{(i)}, \dots, t_q^{(i)})$, we have that  the real numbers
$$
{\rm area}_{g_i}(S_1), \dots, {\rm area}_{g_i}(S_q)
$$
are linearly independent over $\mathbb{Q}$. Sharp's Compactness Theorem, together with the fact that $S_l$ is $g'$-nondegenerate for every $l=1, \dots, q$, implies that for sufficiently large $i$ the only connected $g_i$-minimal hypersurfaces with index at most $p$ and 
area at most $\alpha$ are $S_1, \dots, S_q$. Hence
$g_i \in \mathcal{V} \cap \mathcal{U}_{p, \alpha}$ for sufficiently large $i$. This proves density of $\mathcal{U}_{p, \alpha}$, which finishes the proof of the claim.

\medskip

The claim implies that the set $X=\cap_{N \in \mathbb{N}} \, \mathcal{U}_{N,N}$ is Baire-residual in $\mathcal{M}$, hence it is also dense. This finishes the proof of the Proposition.

\end{proof}

\begin{proof}[Proof of Theorem \ref{theorem.closed.2}]

We denote by $c_j(t)$ the number of elements  $\Sigma \in \tilde{\mathfrak{M}}_g \cap \mathcal{Z}_n(M^{n+1}, \mathbb{Z}_2)$ with $\text{index}(\Sigma)=j$  and $\text{area}(\Sigma)<t$. 

Recall
$$
\mathcal{Z}^t=\{T \in \mathcal{Z}_n(M;{\bf F}; \mathbb{Z}_2): {\bf M}(T) < t\}.
$$

All the homology groups will be computed with coefficients in $\mathbb{Z}_2$. A $k$-dimensional homology class is represented by finite chains of singular $k$-simplices $\sum s_i$, where $s_i:\Delta^k \rightarrow \mathcal{Z}_n(M;{\bf F};\mathbb{Z}_2)$ is a continuous map defined on the standard $k$-simplex $\Delta^k$ for every $i$.

 \subsection{Homology Min-Max Theorem}\label{homology.minmax}
 {\em Let $\sigma \in H_k(\mathcal{Z}^t, \mathcal{Z}^s)$ be a nontrivial homology class, $0\leq s<t$, and let
 $$
 W(\sigma) = \inf_{[\sum s_i]=\sigma}\sup_{i, x\in \Delta^k} {\bf M}(s_i(x)).
 $$
Suppose that, for some $\varepsilon>0$, there is no  $\Sigma'\in \tilde{\mathfrak{M}}_g$ with $\text{area}(\Sigma') \in (s-\varepsilon, s)$ and $\text{index}(\Sigma') \leq k-1$. Then  $W(\sigma)\in [s,t)$ with $W(\sigma)>0$ if $s=0$. Moreover  there exists $\Sigma \in \tilde{\mathfrak{M}}_g \cap \mathcal{Z}_n(M^{n+1}, \mathbb{Z}_2)$ with ${\rm index}(\Sigma)=k$ and
 $$
 {\rm area}(\Sigma) = W(\sigma).
 $$
 }
 
 \begin{proof}
 
 Since there is a representative $[\sum s_i]=\sigma$ with $s_i(\Delta^k)\subset \mathcal{Z}^t$ for each $i$, it is clear that $W(\sigma)<t$. If $W(\sigma)<s$, then there is a representative $[\sum s_i]=\sigma$ with $ s_i(\Delta^k)\subset \mathcal{Z}^s$ for every $i$. This implies $\sigma=0$, which is
 a contradiction. Hence $W(\sigma)\in [s,t)$.

 Let $\{\sum_i s_i^{(j)}\}_j$ be a sequence of representatives ($[\sum_i s_i^{(j)}]=\sigma$)  such that 
 $$
\lim_{j\rightarrow \infty} \sup_{i, x\in \Delta^k} {\bf M}(s_i^{(j)}(x))=W(\sigma).
 $$
 Associated to  the  chain $\sum_i s_i^{(j)}$ we have a $\Delta$-complex $X^{(j)}$ and a map $\Phi^{(j)}: X^{(j)} \rightarrow \cup_i s_i^{(j)}(\Delta^k)$ (see Section 2.1 of \cite{hatcher}) that is continuous in the ${\bf F}$-metric. The boundary $\partial X^{(j)}$ is the union of $(k-1)$-faces of  $\sum_i s_i^{(j)}$ that do not cancel out
 in the calculation of $\partial (\sum_i s_i^{(j)})$. Because the  chain $\sum_i s_i^{(j)}$ is a relative cycle, we have that $\Phi^{(j)}(\partial X^{(j)})\subset \mathcal{Z}^s$. We also have
 $$
 \Phi^{(j)}_*([X^{(j)}]) =\sigma.
 $$

  We can consider the class $\Pi_{\partial}^{(j)}$ of all ${\bf F}$-continuous maps  $\Psi:\partial X^{(j)} \rightarrow \mathcal{Z}^s$ that are  ${\bf F}$-homotopic to $\Phi^{(j)}_{|\partial X^{(j)}}$ through maps taking values in $\mathcal{Z}^s$. By applying
 min-max theory for the area functional to $\Pi_{\partial}^{(j)}$, and using the fact that there is no  $\Sigma'\in \tilde{\mathfrak{M}}_g$ with $\text{area}(\Sigma') \in (s-\varepsilon, s)$ and $\text{index}(\Sigma') \leq k-1$, we conclude that there is an element
 $\Psi'\in \Pi_{\partial}^{(j)}$ with $\sup_{y\in \partial X^{(j)}} {\bf M}(\Psi(y))\leq s-\varepsilon/2$. In particular, we can suppose that
 $$
 \sup_{y\in \partial X^{(j)}} {\bf M}(\Phi^{(j)}(y))\leq s-\varepsilon/2.
 $$
 
 Notice that in Section 3 of \cite{marques-neves-index}, a homotopy class is defined in an unusual way: the homotopy class of an ${\bf F}$-continuous map
 is defined as the class  of all ${\bf F}$-continuous maps that are homotopic to the original one in the flat topology.
But Proposition 3.5 of \cite{marques-neves-infinitely}  has been upgraded to the mass topology in Proposition 3.2 of \cite{marques-neves-lower-bound}.  An inspection of Section 3 of \cite{marques-neves-index} shows that the min-max theory also works for usual ${\bf F}$-continuous homotopy classes. Deformation Theorem A of \cite{marques-neves-index} also applies to to rule out minimal hypersurfaces of higher index.

Since $W(\sigma)\geq s$, we have that the techniques of \cite{marques-neves-lower-bound} for standard
 homology classes can be adapted to the case of relative homology classes as in the statement of the theorem. There is $\delta >0$ such that for sufficiently large $j$,
$$
\sup_{x\in X^{(j)}} {\bf M}(\Phi^{(j)}(x))\leq t-\delta
$$
for every $x\in X^{(j)}$.

Fix $j$. Let $\{g_i\}$ be a sequence of metrics as in Proposition \ref{rational.independence.closed} such that $g_i$ converges to $g$ in the smooth topology. We consider
$$
\sigma_{i}^{(j)} = \Phi^{(j)}_*([X^{(j)}]) \in H_k(\mathcal{Z}_{g_i}^{t-\delta/2},\mathcal{Z}_{g_i}^{s-\varepsilon/4})
$$
 for sufficiently large $i$. Notice that in this case we have
 $$
 \mathcal{Z}_{g_i}^{t-\delta/2} \subset \mathcal{Z}_g^{t}, \, \, \, \mathcal{Z}_{g_i}^{s-\varepsilon/4} \subset \mathcal{Z}_g^s
 $$
Since $\sigma\neq 0$, this implies $\sigma_i^{(j)}\neq 0$ for sufficiently large $i$. Notice that
$$
\left(\sup \frac{g(v,v)}{g_i(v,v)}\right)^{-\frac{n}{2}}W(\sigma) \leq W_{g_i}(\sigma_i^{(j)}) \leq \sup_{x\in X^{(j)}} {\bf M}_{g_i}(\Phi^{(j)}(x)),
$$
hence
$$
W(\sigma) \leq \liminf_{i\rightarrow \infty}W_{g_i}(\sigma_i^{(j)})\leq  \limsup_{i\rightarrow \infty}W_{g_i}(\sigma_i^{(j)})\leq  \sup_{x\in X^{(j)}} {\bf M}_{g}(\Phi^{(j)}(x)).
$$

Notice that there is no $g_i$-minimal hypersurface with index less than or equal to $k-1$ and area in $(s-3\varepsilon/4,s-\varepsilon/4)$ for large $i$. Hence, by the same argument as before, we can find a sequence
$$
\Phi_l^{i,(j)}: X_l^{i,(j)}\rightarrow \mathcal{Z}_{g_i}^{t-\delta/2}, \, \, \, \Phi_l^{i,(j)}(\partial X_l^{i,(j)}) \subset \mathcal{Z}_{g_i}^{s-\varepsilon/2} \subset \mathcal{Z}_{g_i}^{s-\varepsilon/4},
$$
such that
$$
(\Phi_l^{i,(j)})_*([X_l^{i,(j)}]) =\sigma_i^{(j)}
$$
and
$$
\lim_{l \rightarrow \infty} \sup_{x\in X_l^{i,(j)}}  {\bf M}_{g_i}(\Phi_l^{i,(j)}(x))=W_{g_i}(\sigma_i^{(j)}).
$$

Let $\Pi_l^{i,(j)}$ be the homotopy class of $\Phi_l^{i,(j)}:X_l^{i,(j)} \rightarrow  \mathcal{Z}_{g_i}^{t-\delta/2}$ relative to $\partial X_l^{i,(j)}$. If
 $$
 {\bf L}(\Pi_l^{i,(j)}) =\inf_{\Phi \in \Pi_l^{i,(j)}}\sup_{x\in X_l^{i,(j)}} {\bf M}_{g_i}(\Phi(x)),
 $$
 then
 $$
 W_{g_i}(\sigma_i^{(j)}) \leq  {\bf L}(\Pi_l^{i,(j)})\leq \sup_{x\in X_l^{i,(j)}}  {\bf M}_{g_i}(\Phi_l^{i,(j)}(x)).
 $$
 In particular, $\lim_{l\rightarrow \infty} {\bf L}(\Pi_l^{i,(j)})=W_{g_i}(\sigma_i^{(j)})$.
 

Min-max theory gives a stationary integral varifold $V_l^{i,(j)}$ whose support is a $g_i$-minimal hypersurface such that 
$\text{area}_{g_i}(V_l^{i,(j)})={\bf L}(\Pi_l^{i,(j)})$. Deformation Theorem A of \cite{marques-neves-index} allows us to rule out minimal hypersurfaces of higher index, so that we can choose $V_l^{i,(j)}$ to satisfy $\text{index}(V_l^{i,(j)})\leq k$. Here $\text{index}(V_l^{i,(j)})$ is defined to be the sum of the indices of the components of $\text{spt}(V_l^{i,(j)})$. Because the metric is bumpy, for sufficiently large $l$ we have 
$${\bf L}(\Pi_l^{i,(j)})=\text{area}_{g_i}(V_l^{i,(j)})=W_{g_i}(\sigma_i^{(j)}).$$

Moreover, the work of X. Zhou \cite{zhou-multiplicity} gives a minimizing sequence $\{\Psi_m\}_m$ in $\Pi_l^{i,(j)}$ so that the critical set ${\bf C}(\{\Psi_m\}_m)$ contains a multiplicity one, minimal hypersurface $\Sigma_l^{i,(j)} \in \mathcal{Z}_n(M^{n+1}, \mathbb{Z}_2)$ such that
${\rm index}(\Sigma_l^{i,(j)})\leq k$. Because $g_i$ is as in Proposition \ref{rational.independence.closed}, $\Sigma_l^{i,(j)}$ is the unique smooth element of
${\bf C}(\{\Psi_m\}_m)$.  
Theorem 4.7 (i) of \cite{marques-neves-lower-bound} can be upgraded to homotopy in the ${\bf F}$-metric. The minimizing sequence produced by Theorem 4.9 of \cite{marques-neves-lower-bound} is in the usual ${\bf F}$-metric homotopy class.
The proof of Theorem 7.2 of \cite{marques-neves-lower-bound} can also be adapted to our setting, which proves that 
$\text{index}(\Sigma_l^{i,(j)})=k$. The only point to check is that the singular cycle $z_{i,p}$  of the proof of Theorem 7.2 of \cite{marques-neves-lower-bound} is null-homologous in the ${\bf F}$-metric. But this is a consequence of Theorem 3.8 of \cite{marques-neves-lower-bound}
with $\mathcal{K}=\tilde{\Sigma}_{q_p}$, in the notation of the proof of Theorem 7.2 of \cite{marques-neves-lower-bound}. 

We have found a multiplicity one, $g_i$-minimal hypersurface $\Sigma^{i,(j)}$ such that
$$
\text{index}(\Sigma^{i,(j)})=k, \, \, \, \text{area}_{g_i}(\Sigma^{i,(j)})=W_{g_i}(\sigma_i^{(j)}).
$$

If we let $i\rightarrow \infty$, and because $g$ is bumpy, we can take a subsequential limit of $\Sigma^{i,(j)}$ and get a multiplicity one $g$-minimal hypersurface $\Sigma^{(j)}\in \mathcal{Z}_n(M^{n+1}, \mathbb{Z}_2)$ such that
$$
\text{index}(\Sigma^{(j)})=k, \, \, \, W(\sigma) \leq \text{area}_g(\Sigma^{(j)}) \leq \sup_{x\in X^{(j)}} {\bf M}_{g}(\Phi^{(j)}(x)).
$$
Taking another subsequential limit as $j\rightarrow \infty$, we find a multiplicity one $g$-minimal hypersurface $\Sigma \in \mathcal{Z}_n(M^{n+1}, \mathbb{Z}_2) $ such that
$$
\text{index}(\Sigma)=k, \, \, \,   \text{area}_g(\Sigma) =W(\sigma).
$$

If $s=0$ and $W(\sigma)=0$, then there are representatives of the homology class  $\sigma$ with supremum of the masses arbitrarily small. Then Theorem 3.8 of
\cite{marques-neves-lower-bound} (with $\mathcal{K}=\{0\}$ and $\Psi\equiv 0$)  implies that $\sigma=0$, which gives a contradiction.
This finishes the
proof of the Min-Max Homology Theorem.

 \end{proof}
 

\subsection{Proposition}\label{betti.prop} {\em Suppose that  the metric $g$ is as in Proposition \ref{rational.independence.closed}. Let  $r \in \mathbb{Z}_+$ and $\Sigma \in  \tilde{\mathfrak{M}}_g \cap  \mathcal{Z}_n(M^{n+1}, \mathbb{Z}_2)$ with $\text{index}(\Sigma)=k\leq r$ and $\text{area}(\Sigma)=a$. For every $\varepsilon>0$, there exists $a-\varepsilon < s< a< t < a+\varepsilon$
such that
\begin{eqnarray*}
&&b_i(\mathcal{Z}^{t},\mathcal{Z}^{s}) =0 \, \, \forall \, i \leq r, i\neq k,\\
&&b_k(\mathcal{Z}^{t},\mathcal{Z}^{s}) =1.
\end{eqnarray*}
}
  
  \begin{proof}
  Because the metric $g$ is as in Proposition \ref{rational.independence.closed}, there is $\delta>0$ such that $\Sigma$ is the unique element of 
  $\tilde{\mathfrak{M}}_g$ with index less than or equal to $r$ and area in $(a-\delta,a+\delta)$.

  Let $\sigma  \in H_i(\mathcal{Z}^t,\mathcal{Z}^s)$ with $i \leq r, i\neq k$, $a-\delta < s<t<a+\delta$. If $\sigma \neq 0$, the Homology Min-Max Theorem gives an element
  $\Sigma' \in \tilde{\mathfrak{M}}_g$ with $\text{index}(\Sigma')=i$ and $\text{area}(\Sigma') \in [s,t)$. Contradiction, hence $\sigma=0$. This proves
  $$b_i(\mathcal{Z}^t,\mathcal{Z}^s) =0 \, \, \forall \, i \leq r, i\neq k, \, \, a-\delta < s<t<a+\delta.$$
  

  Consider $\{\eta_i\}_{i=1}^k$ and the functional $A^*$   as in the proof of the Local Min-Max Theorem 6.1 of \cite{marques-neves-lower-bound}. Let $\varepsilon_1$ be as in
  Proposition 6.2 of \cite{marques-neves-lower-bound}. Then for every $S\in \mathcal{Z}_n(M;\mathbb{Z}_2)$ with $\mathcal{F}(S,\Sigma) < \varepsilon_1$, we have $A^*(S)>A^*(\Sigma)$ unless $S=\Sigma$. This  implies that for each $\alpha>0$, there exists $\kappa=\kappa(\alpha, \Sigma,\varepsilon_1)>0$ such that
  if $S\in \mathcal{Z}_n(M;\mathbb{Z}_2)$ satisfies $\mathcal{F}(S,\Sigma) \leq \varepsilon_1/2$ and $A^*(S)\leq A^*(\Sigma)+\kappa$, then
  $\mathcal{F}(S,\Sigma) \leq \alpha$.
 
  Choose $t\in (a,a+\varepsilon)$, with $t-a <\min\{\delta, \kappa(\varepsilon_1/16,\Sigma,\varepsilon_1)\}$. If $\mathcal{K}_1=\{\Sigma\}$, let $\tilde \beta=\beta(\mathcal{K}_1, \frac{t-a}{4})>0$ as in Theorem 3.8
  of \cite{marques-neves-lower-bound}.
  
   Given $\beta>0$, let $\varepsilon_0>0$ and  $\{F_v\}_{v\in \overline{B}^k}$ be as in  Theorem 6.1 of \cite{marques-neves-lower-bound}. We can suppose, by choosing $\beta>0$ sufficiently small,  that
    $$
    {\bf F}((F_v)_{\#}(\Sigma),\Sigma)\leq \min\{\varepsilon_1/4, \tilde{\beta}/4 \}
    $$ for all $v \in \overline{B}^k$. If 
   $$
   P^\Sigma(v) = \sum_{i=1}^k ||(F_v)_{\#}\Sigma||(\eta_i)\cdot e_i,
   $$
   we have $P^\Sigma(0)=0$ and $DP^\Sigma(0)= \text{Id}$. Hence 
   $$
   (P^\Sigma)_*([\partial \overline{B}^k]) \neq 0 \in H_{k-1}(\R^k \setminus \{0\}).
   $$
   We can choose $0<\gamma<\min \{\delta/2,   (t-a)/4, \tilde\beta\}$ such that 
   $$
   \text{area}\left((F_v)_{\#}(\Sigma)\right)< \text{area}(\Sigma)-2\gamma
   $$
   for every $v \in \partial B^k$.
    If
   $$
   \mathcal{K}_2=\{(F_v)_{\#}(\Sigma): v\in \overline{B}^k\},
   $$
   let $0<\beta'=\beta(\mathcal{K}_2, \gamma/4)<\gamma/4$ as in Theorem 3.8 of \cite{marques-neves-lower-bound}.
   
   We choose $0<\overline h<\tilde\beta/4$, depending only on $\Sigma$ and $\{F_v\}$, such that for every $S\in \mathcal{Z}_n(M;\mathbb{Z}_2)$ with ${\bf F}(S,\Sigma)<\overline h$, we have that the function
   $$
   v\in \overline{B}^k \mapsto  ||(F_v)_{\#}(S)||(M)
   $$
   is strictly concave and has a unique maximum point somewhere in $B_{1/2}^k(0)$, and
   $$
   {\bf F}((F_v)_{\#}(S), (F_v)_{\#}(\Sigma)) < \min\{\gamma/4, \beta'/4, \tilde{\beta}/4\}
   $$
   for every $v \in \overline{B}^k$.
  
  If $\overline{\Phi}:\overline{B}^k \rightarrow \mathcal{Z}_n(M;{\bf F};\mathbb{Z}_2)$  is given by 
  $\overline{\Phi}(v) = (F_v)_{\#}(\Sigma)$, we define the relative homology class
  $$
  \overline{\sigma} = \overline{\Phi}_*([\overline{B}^k]) \in H_k(\mathcal{Z}^{t}, \mathcal{Z}^{a-\gamma}),
  $$
  where $[\overline{B}^k]$ denotes the generator of    $H_k(\overline{B}^k, \partial \overline B^k).$
  
  \subsection{Claim} {\em $\overline{\sigma}\neq 0$.}
  \medskip
  
  Suppose $\overline{\sigma}=0$ by contradiction. This implies that there exists a $(k+1)$-dimensional $\Delta$-complex $Y$
  and a map $H: Y \rightarrow \mathcal{Z}^{t}$, continuous in the ${\bf F}$-metric, such that $\partial Y$ is the union of two $k$-dimensional $\Delta$-subcomplexes
  $$
  \partial Y = \overline{B}^k \cup Z,
  $$
  disjoint as subcomplexes, such that
  $$
  H_{|\overline{B}^k} = \overline{\Phi},
  $$
  $$
  \sup_{y\in Z} {\bf M}(H(y)) < a-\gamma.
  $$
  
  We can apply barycentric subdivision to $Y$ (still denoted by $Y$) so that for each $(k+1)$-dimensional simplex $s\in Y$ we have
  $$
  \sup_{y_1,y_2\in s}{\bf F}(H(y_1),H(y_2)) < \varepsilon_1/16.
  $$
  Let $W$ be the union of all $(k+1)$-simplices $s\in Y$ such that there exists $y\in s$ with
  $$
  \mathcal{F}(H(y),\Sigma) \leq \varepsilon_1/4.
  $$
  In particular, $\overline{B}^k \subset \partial W$. Also, 
  $$
  \mathcal{F}(H(y),\Sigma)\leq 3\varepsilon_1/8
  $$
  for every $y\in W$.

 Consider a $k$-simplex $t \in \partial W$ such that $t \notin \overline{B}^k \cup Z$. Then there exists $s\in Y$ with $s\notin W$
 and $t \subset \partial s$. By definition of $W$, we have
 $$
  \mathcal{F}(H(y),\Sigma) >\varepsilon_1/4
 $$
 for every $y\in s$. Hence, for any $y'\in t$ and $y\in s$ we have
 \begin{eqnarray*}
&&\mathcal{F}(H(y'),\Sigma)\geq \mathcal{F}(H(y),\Sigma)-\mathcal{F}(H(y'),H(y))\\
&&> \varepsilon_1/4-\varepsilon_1/16 \geq \varepsilon_1/8.
 \end{eqnarray*}

 We can decompose $\partial W$ as the union of two $k$-dimensional subcomplexes, disjoint as subcomplexes:
 $$
 \partial W = \overline{B}^k \cup B'.
 $$
 Note that $\partial B'=\partial \overline{B}^k$. Consider the continuous map $P: B'\rightarrow \R^k$ given by 
 $$
 P(y) = \sum_{i=1}^k ||H(y)||(\eta_i) \cdot e_i.
 $$
 Note that $P_{|\partial B'} =P^\Sigma_{|\partial \overline{B}^k}$. Therefore $P_*([\partial B'])\neq 0 \in H_{k-1}(\R^k \setminus \{0\})$. This implies the
 existence of $y_0\in B'$ such that $P(y_0)=0$. So
 $$
 A^*(H(y_0))=||H(y_0)||(M)={\bf M}(H(y_0)).
 $$

 This implies 
\begin{equation}\label{a*}
 A^*(H(y_0)) < t<a +\kappa(\varepsilon_1/16,\Sigma, \varepsilon_1).
 \end{equation}
 Since $y_0\in W$, then $\mathcal{F}(H(y_0),\Sigma) \leq 3\varepsilon_1/8<\varepsilon_1/2$ and hence
 $$
 {\bf M}(H(y_0)) = A^*(H(y_0))\geq A^*(\Sigma)= {\bf M}(\Sigma)=a.
 $$
 This gives $y_0\notin Z$. Hence $\mathcal{F}(H(y_0),\Sigma) \geq \varepsilon_1/8.$ On the other hand, inequality (\ref{a*}) gives $\mathcal{F}(H(y_0),\Sigma)\leq \varepsilon_1/16$. Contradiction, thus $\overline{\sigma}\neq 0$ which proves the claim.

 \medskip
 
 \subsection{Claim} {\em $H_k(\mathcal{Z}^{t}, \mathcal{Z}^{a-\gamma})=\{0, \overline{\sigma}\}$.}
  
\medskip

      Let  $\sigma \in H_k(\mathcal{Z}^{t},\mathcal{Z}^{a-\gamma})$, with $\sigma\neq 0$. The Homology Min-Max Theorem gives that $W(\sigma)=\text{area}(\Sigma)$. Like in the proof of the Homology Min-Max Theorem, there exists a homotopy class $\Pi$ of maps
      $\Phi:X \rightarrow \mathcal{Z}^{t}$ relative to $\partial X$, where $X$ is a $k$-dimensional $\Delta$-complex and $\Phi(\partial X)\subset \mathcal{Z}^{a-\gamma}$, such that ${\bf L}(\Pi)=\text{area}(\Sigma)$ and $\Phi_*([X])=\sigma$. Let $\{\Phi_j\}\subset \Pi$ be a minimizing sequence. Because the metric is as in
      Proposition \ref{rational.independence.closed}, the only smooth element of ${\bf C}(\{\Phi_j\})$ is $\Sigma$. Given any  $0<h<\min\{\overline{h}/2, \varepsilon_0/2\}$,
      the proof of  Theorem 7.2 of \cite{marques-neves-lower-bound} gives another minimizing sequence $\{\hat\Phi_j\}\subset \Pi$ (which can be taken continuous in the mass topology) such that the only smooth element
      of ${\bf C}(\{\hat \Phi_j\})$ is $\Sigma$ and for some $\eta>0$ we have
      $$
    x\in X,\,   {\bf M}(\hat \Phi_j(x))\geq a-\eta \implies {\bf F}(|\hat \Phi_j(x)|,|\Sigma|)<h.
      $$
      Because of Theorem 4.11 of \cite{marques-neves-lower-bound}, we can also suppose
      $$
    x\in X,\,   {\bf M}(\hat \Phi_j(x))\geq a-\eta \implies \mathcal{F}(\hat \Phi_j(x),\Sigma)<h.
      $$
      We can assume $\eta <\gamma/2$.

      We fix $j$ and choose  $0<h'<\min\{\eta/8, \gamma/4, \beta'/4\}$ and a barycentric subdivision of $X$ (still denoted by $X$) such that for each $k$-simplex $t \in X$, we have
      $$
      {\bf F}(\hat \Phi_j(x_1), \hat \Phi_j(x_2)) < h'
      $$
      for every $x_1,x_2\in t$. Let $V$ be the union of all $k$-simplices $t\in X$ such that there exists $x\in t$ with
      $$
      {\bf M}(\hat \Phi_j(x))\geq a-\eta/4.
      $$
      Then, for any other $x'\in t$, we have
\begin{eqnarray*}
&&{\bf M}(\hat \Phi_j(x'))\geq {\bf M}(\hat \Phi_j(x))-{\bf F}(|\hat \Phi_j(x')|, |\hat \Phi_j(x)|)\\
&&\geq a-\eta/4-h'\\
&&\geq a-3\eta/8.
       \end{eqnarray*}
       In particular, $V\cap \partial X=\emptyset$. Now, for any $(k-1)$-simplex $s\in \partial V$ there exists a $k$-simplex  $\overline t\notin V$ with $s \subset \partial \overline t$. By definition of $V$, we have:
       \begin{equation}\label{mass.bound}
       {\bf M}(\hat \Phi_j(y))\leq a-\eta/4
       \end{equation}
       for every $y\in \partial V$.


       We will follow as in the proof of Theorem 7.2 of \cite{marques-neves-lower-bound}. For $y \in \partial V$, let $A^y:\overline{B}^k \rightarrow [0,\infty)$ be defined by
       $$
       A^y(v) = ||(F_v)_{\#}(|\hat \Phi_j(y)|)||(M).
       $$
      Note that
      $$
      {\bf F}(\hat \Phi_j(y), \Sigma) < 2h <\overline h
      $$
      for every $y\in V$.
      Hence the function $A^y$ is strictly concave and has a unique maximum at $m(y)\in B^k_{\frac12}(0)$. By Theorem 6.1 of \cite{marques-neves-lower-bound} and inequality (\ref{mass.bound}) above, we get that $m(y)\neq 0$ for every $y\in \partial V$. We can find $\zeta>0$ such that $\zeta\leq |m(y)|<1/2$ for every $y\in \partial V$.

      Now we can consider the one-parameter flow $\{\phi^y(\cdot,t)\}_{t\geq 0} \subset \text{Diff}(\overline{B}^k)$ generated by the vector field
      $$
      v \mapsto -(1-|v|^2)\nabla A^y(v).
      $$
      The function $t \mapsto A^y(\phi^y(v,t))$ is nonincreasing for fixed $v$. It is strictly decreasing except if $v=m(y)$ or if $v\in \partial B^k$, in which cases it is constant. We have that
      $$
      \lim_{t \rightarrow \infty}\phi^y(v,t)\in \partial B^k
      $$
      if $v\neq m(y)$, and the limit is uniform if $|v-m(y)|\geq \zeta$.

      Therefore there exists $s>0$ such that the homotopy
      $$
      H_1: [0,1] \times \partial V \rightarrow \mathcal{Z}_n(M;{\bf F};\mathbb{Z}_2)
      $$
      defined by
      $$
     H_1(t,y)=(F_{\phi^y(0,st)})_{\#}(\hat \Phi_j(y))
      $$
      satisfies:
      \begin{itemize}
      \item $H_1(0,y)=\hat \Phi_j(y)$,
      \item ${\bf F}( H_1(1,y), (F_{w(y)})_{\#}(\hat \Phi_j(y)))<h'$ for some continuous function
      $$
      w:\partial V \rightarrow \partial B^k,
      $$
      \item ${\bf M}(H_1(t,y))\leq a-\eta/8$ 
      \end{itemize}
      for every $(t,y)\in  [0,1] \times \partial V$. 
      We have
      \begin{eqnarray*}
     && {\bf F}(H_1(1,y), (F_{w(y)})_{\#}(\Sigma))\leq  {\bf F}(H_1(1,y), (F_{w(y)})_{\#}(\hat \Phi_j(y)))\\
   &&\hspace{6cm}   + {\bf F}((F_{w(y)})_{\#}(\hat \Phi_j(y)), (F_{w(y)})_{\#}(\Sigma))\\
   &&\leq h'+\beta'/4 \leq \beta'/2
      \end{eqnarray*}     
       for every $y\in \partial V$. In particular,
       \begin{eqnarray*}
      && {\bf M}(H_1(1,y))\leq {\bf M}(F_{w(y)})_{\#}(\Sigma)) + {\bf F}(H_1(1,y), (F_{w(y)})_{\#}(\Sigma))\\
      &&\leq a-2\gamma+\beta'/2\leq a -2\gamma+\gamma/8\leq a-15\gamma/8.
       \end{eqnarray*}
         
          We also have
          \begin{eqnarray*}
         && {\bf F}(H_1(t,y), \Sigma)={\bf F}((F_{\phi^y(0,st)})_{\#}(\hat \Phi_j(y)),\Sigma)\\
         &&\leq {\bf F}((F_{\phi^y(0,st)})_{\#}(\hat \Phi_j(y)),(F_{\phi^y(0,st)})_{\#}(\Sigma))+ {\bf F}((F_{\phi^y(0,st)})_{\#}(\Sigma), \Sigma)\\
         &&\leq \tilde\beta/4+\tilde\beta/4=\tilde\beta/2
          \end{eqnarray*}
            for every $(t,y)\in  [0,1] \times \partial V$.

      By applying Theorem 3.8 of \cite{marques-neves-lower-bound} with $\mathcal{K}=\mathcal{K}_2$, we get that there exists a second homotopy
      $$
      H_2:[1,2] \times \partial V \rightarrow \mathcal{Z}_n(M;{\bf F};\mathbb{Z}_2)
      $$
      such that
      \begin{itemize}
      \item $H_2(1,y)=H_1(1,y),$
      \item $H_2(2,y)=(F_{w(y)})_{\#}(\Sigma)$,
      \item ${\bf F}(H_2(t,y),(F_{w(y)})_{\#}(\Sigma))\leq \gamma/4$
      \end{itemize}
      for every $(t,y)\in [1,2]\times \partial V$. In particular,
      \begin{eqnarray*}
      &&{\bf F}(H_2(t,y),\Sigma)\leq {\bf F}(H_2(t,y),(F_{w(y)})_{\#}(\Sigma))+ {\bf F}((F_{w(y)})_{\#}(\Sigma), \Sigma)\\
      &&\leq \gamma/4+\tilde\beta/4\leq \tilde\beta/2,
      \end{eqnarray*}
      and
      \begin{eqnarray*}
     && {\bf M}(H_2(t,y))\leq {\bf M}((F_{w(y)})_{\#}(\Sigma)) + {\bf F}(H_2(t,y),(F_{w(y)})_{\#}(\Sigma))\\
      &&\leq a-2\gamma +\gamma/4\leq a-7\gamma/4
      \end{eqnarray*}
      for every $(t,y)\in [1,2]\times \partial V$.

      
      We take $Q$ to be the cone over $\partial V$ and $\hat w: Q\rightarrow \overline B^k$ to be a continuous map that
     sends the vertex of the cone to $0$ and such that $\hat w(y)=w(y)$ for every $y\in \partial V=\partial Q$.
     
      
      Consider the following $\Delta$-complex:
      \begin{eqnarray*}
      C = V \cup ([0,2]\times \partial V) \cup Q,
      \end{eqnarray*}
      so $\partial C=0.$ Define the continuous map
      $$
      \Psi: C \rightarrow \mathcal{Z}_n(M;{\bf F};\mathbb{Z}_2)
      $$
      by
      $$
      \Psi(x)=\hat\Phi_j(x)
      $$
      for $x\in V$,
      $$
      \Psi(t,y) = H_1(t,y)
      $$
      for $(t,y)\in [0,1] \times \partial V$,
      $$
      \Psi(t,y) = H_2(t,y)
      $$
      for $(t,y)\in [1,2] \times \partial V$, and
      $$
      \Psi(q) = (F_{\hat w(q)})_{\#}(\Sigma)
      $$
      for $q\in Q$. Then
      $$
      {\bf F}(\Psi(p),\Sigma) \leq \tilde\beta/2
      $$
      for every $p\in C$. 
      
      By applying Theorem 3.8 of \cite{marques-neves-lower-bound} with $\mathcal{K}=\mathcal{K}_1=\{\Sigma\}$, there exists a homotopy
      $$
      H_3:[0,1] \times C \rightarrow  \mathcal{Z}_n(M;{\bf F};\mathbb{Z}_2)
      $$
      with
      \begin{itemize}
      \item $H_3(0,p)=\Psi(p)$,
      \item $H_3(1,p)=\Sigma$,
      \item ${\bf F}(H_3(t,p),\Sigma)\leq (t-a)/4$
      \end{itemize}
      for every $(t,p)\in [0,1]\times C$. In particular,
     \begin{eqnarray*}
     && {\bf M}(H_3(t,p))\leq {\bf M}(\Sigma) + {\bf F}(H_3(t,p),\Sigma)\leq a + (t-a)/4=t-\frac34(t-a)\\
    && <t
    \end{eqnarray*}
    for every $(t,p)\in [0,1]\times C$. Hence
    $$
    \Psi_*([C]) =0 \in H_k(\mathcal{Z}^t, \mathcal{Z}^{a-\gamma}).
    $$

     Now consider the $\Delta$-complex
     $$
     X' = (X\setminus V)  \cup ([0,2]\times \partial V) \cup Q,
     $$ 
     and the map
     $$
     \Psi':X' \rightarrow \mathcal{Z}_n(M;{\bf F};\mathbb{Z}_2)
     $$
     given by
     $$
     \Psi'(x)=\hat\Phi_j(x)
     $$
     for every $x\in X\setminus V$, and
     $$
     \Psi'(p)=\Psi(p)
     $$
     for every $p\in ([0,2]\times \partial V) \cup Q$. So 
     $$
     \Psi'_*([X'])=\sigma \in  H_k(\mathcal{Z}^t, \mathcal{Z}^{a-\gamma}).
     $$
     Note that
     ${\bf M}(\Psi'(z))\leq a-\eta/8$ for every $z\in X'\setminus Q$ and ${\bf M}(\Psi'(p))\leq a-2\gamma$ for every $p\in \partial Q$. 
     
     Let $\Pi'$ be the homotopy class of the map
     $$
     \Psi'_{|(X'\setminus Q)}: X'\setminus Q\rightarrow \mathcal{Z}^{a-\eta/16}
     $$      
     relative to the boundary $\partial (X'\setminus Q) =\partial X  \cup \partial Q$. Note that
     $$
     \Psi'(\partial (X'\setminus Q)) \subset \mathcal{Z}^{a-\gamma}.
     $$
     By applying min-max theory  for the area functional to $\Pi'$, and using that there is no minimal hypersurface with index less than or equal to $k$ and area in $[a-\gamma,a-\eta/16)$, we get that there exists $\Psi''\in \Pi'$ with
     $$
     {\bf M}(\Psi''(z))< a-\gamma
     $$
     for every $z\in X'\setminus Q$. This implies
     $$
     \Psi_*([Q]) =\sigma \in H_k(\mathcal{Z}^t, \mathcal{Z}^{a-\gamma}).
     $$
    
     Recall that $\overline  \Phi:\overline B^k \rightarrow \mathcal{Z}(M;{\bf F};\mathbb{Z}_2)$ is defined by $\overline \Phi(v)=(F_v)_{\#}(\Sigma)$ and $\overline\sigma = \overline \Phi_*([\overline B^k])\in H_k(\mathcal{Z}^t, \mathcal{Z}^{a-\gamma})$. Then
     $$
     \Psi_{|Q}=\overline \Phi \circ \hat w,
     $$
where $\hat w:(Q,\partial Q)\rightarrow  (\overline B^k, \partial B^k)$ and $\overline \Phi:(\overline B^k, \partial B^k)\rightarrow (\mathcal{Z}^t, \mathcal{Z}^{a-\gamma})$.

Hence
$$
\sigma = \Psi_*([Q]) = \overline \Phi_*(d),
$$
with $d=\hat w_*([Q]) \in H_k(\overline{B}^k,\partial B^k)=\mathbb{Z}_2$. Since $\sigma \neq 0$, we have that $d=[\overline B^k]$. Hence
$$
\sigma=\overline{\sigma},
$$
which finishes the proof of the claim and of the proposition.

  \end{proof}
  
  We can now prove Theorem \ref{theorem.closed.2} for Riemannian metrics as in Proposition \ref{rational.independence.closed}.  Suppose $g$ is such a metric.
  Fix $r\in \mathbb{Z}_+$.   
  Let $\{\Sigma_1, \dots, \Sigma_q\}$ be the collection of elements of $\tilde{\mathcal{M}}_g \cap \mathcal{Z}_n(M^{n+1}, \mathbb{Z}_2)$ with area
  less than $a$ and Morse index less than or equal to $r$. We can suppose that they are ordered so that
  $$
  \text{area}(\Sigma_q) < \text{area}(\Sigma_{q-1}) < \cdots < \text{area}(\Sigma_1).
  $$
  
  Set $a_i=\text{area}(\Sigma_i)$ and $m_i=\text{index}(\Sigma_i)$. Proposition \ref{betti.prop} implies that we can choose $s_i<a_i<t_i$ for every $1\leq i\leq q$ such that
  $t_i<s_{i-1}$ for all $2\leq i \leq q$, $0<s_q<t_1<a$ and
  \begin{eqnarray*}
&&b_j(\mathcal{Z}^{t_i},\mathcal{Z}^{s_i}) =0 \, \, \forall \, j \leq r, j\neq m_i,\\
&&b_{m_i}(\mathcal{Z}^{t_i},\mathcal{Z}^{s_i}) =1.
\end{eqnarray*}
The Homology Min-Max Theorem \ref{homology.minmax} implies
$$
b_j(\mathcal{Z}^{a}, \mathcal{Z}^{t_1}) =0 \, \, \forall \, j \leq r,
$$
$$
b_j(\mathcal{Z}^{s_{i-1}}, \mathcal{Z}^{t_i}) =0 \, \, \forall \, j \leq r, \, \, 2\leq i \leq q,
$$
and
$$
b_j(\mathcal{Z}^{s_q}, \mathcal{Z}^0=\emptyset) =0 \, \, \forall \, j \leq r.
$$

The Betti numbers with coefficients in $\mathbb{Z}_2$ are subadditive, 
meaning that if $X \supset Y \supset Z$ are topological spaces, then
$$
b_k(X,Z) \leq b_k(X,Y)+b_k(Y,Z).
$$
Hence
\begin{eqnarray*}
b_k(\mathcal{Z}^a)&\leq& b_k(\mathcal{Z}^a,\mathcal{Z}^{t_1})+\sum_{i=1}^q b_k(\mathcal{Z}^{t_i},\mathcal{Z}^{s_i})+\sum_{i=2}^q b_k(\mathcal{Z}^{s_{i-1}}, \mathcal{Z}^{t_i}) +b_k(\mathcal{Z}^{s_q})\\
&=& c_k(a) <\infty.
\end{eqnarray*}
Therefore $b_k(a)<\infty$ for every $k\leq r$.

The function 
$$
s_k(X,Y) = b_k(X,Y)-b_{k-1}(X,Y) +\cdots +(-1)^kb_0(X,Y)
$$
is also subadditive. 
Hence
\begin{eqnarray*}
s_k(\mathcal{Z}^a)&\leq& s_k(\mathcal{Z}^a,\mathcal{Z}^{t_1})+\sum_{i=1}^q s_k(\mathcal{Z}^{t_i},\mathcal{Z}^{s_i})+\sum_{i=2}^q s_k(\mathcal{Z}^{s_{i-1}}, \mathcal{Z}^{t_i}) +s_k(\mathcal{Z}^{s_q})\\
&=& \sum_{i=1}^q s_k(\mathcal{Z}^{t_i},\mathcal{Z}^{s_i}) =\sum_{1\leq i\leq q, m_i\leq k} (-1)^{m_i-k}\\
&=& c_k(a)-c_{k-1}(a)+ \cdots +(-1)^kc_0(a).
\end{eqnarray*}
Therefore
$$
b_k(a)-b_{k-1}(a)+ \cdots +(-1)^kb_0(a)\leq c_k(a)-c_{k-1}(a)+ \cdots +(-1)^kc_0(a)
$$
for every $k\leq r$. Since $r\in\mathbb{Z}_+$ is arbitrary, this establishes the Theorem for metrics as in Proposition \ref{rational.independence.closed}.

Let $g$ be a bumpy metric.  Fix $r \in \mathbb{Z}_+$. There are finitely many elements of $\tilde{\mathfrak{M}}_g$ with area less than or equal to $(a+1)$ and Morse index less than or equal to $(r+1)$.  We can  suppose that there is no element of $\tilde{\mathfrak{M}}_{g}$ with index less than or equal to $(r+1)$ and area in $[a-2\delta,a)$ for some $\delta>0$.  Let  $\{g_i\}_i$ be a sequence of metrics as in Proposition \ref{rational.independence.closed} such that 
$g_i$ converges to $g$ in the smooth topology. For sufficiently large $i$,  there is no element of $\tilde{\mathfrak{M}}_{g_i}$ with index less than or equal to $(r+1)$ and area in $[a-2\delta,a-\delta]$. We are using Sharp's compactness theorem \cite{sharp}.



  Let  $\Phi:(X,\partial X) \rightarrow (\mathcal{Z}_g^a,\mathcal{Z}_{g_i}^{a-\delta})$ be  a continuous map, such that $\text{dim}(X)=q \leq (r+1)$, where $i$ is sufficiently large.
  Of course $\mathcal{Z}_{g_i}^{a-\delta} \subset \mathcal{Z}_g^a$ for sufficiently large $i$. We can consider $\Pi_{i,\partial}$ the homotopy class of $\Phi_{|\partial X}$ among all maps taking values in $\mathcal{Z}_{g_i}^{a-\delta}$. Min-max theory in the metric $g_i$ applied to $\Pi_{i,\partial}$, with upper Morse index bounds, gives that there exists $\Psi_\partial \in \Pi_{i,\partial}$ such that $\Psi_\partial (\partial X)\subset \mathcal{Z}_{g_i}^{a-2\delta}$. This means that if we are interested in $\Phi_*([X])\in H_q(\mathcal{Z}_g^a,\mathcal{Z}_{g_i}^{a-\delta})$, we can suppose  $\Phi(\partial X) \subset \mathcal{Z}_{g_i}^{a-2\delta}$. Let $\Pi$ be  the homotopy class of $\Phi$ relative to $\partial X$, among maps that take values in $\mathcal{Z}_g^a$. Min-max theory in the $g$ metric applied to $\Pi$, with upper Morse index bounds, implies the
  existence of $\Psi\in\Pi$ such that $\Psi(X)\subset \mathcal{Z}_g^{a-3\delta/2}$. 
  
  Since $\mathcal{Z}_g^{a-3\delta/2}\subset \mathcal{Z}_{g_i}^{a-\delta}$ for sufficiently large $i$, we conclude that
 $$\Phi_*([X])=0\in H_q(\mathcal{Z}_g^a,\mathcal{Z}_{g_i}^{a-\delta}).$$
  Therefore
  $$
  H_q(\mathcal{Z}_g^a,\mathcal{Z}_{g_i}^{a-\delta})=0
  $$
  for every $q\leq (r+1)$. The long exact homology sequence of the pair $(\mathcal{Z}_g^a,\mathcal{Z}_{g_i}^{a-\delta})$ implies that
  $$
  H_k(\mathcal{Z}_g^a)=H_k(\mathcal{Z}_{g_i}^{a-\delta})
  $$
  for every $k\leq r.$ In particular, $b_k(\mathcal{Z}_g^a)=b_k(\mathcal{Z}_{g_i}^{a-\delta})<\infty$ for every $k\leq r$.
  
  Notice that the way we constructed $g_i$ implies that for sufficiently large $i$ we have 
  $$
  c_{k,g}(a) = c_{k,g_i}(a-\delta)
  $$
  for every $k\leq r$. We are using Sharp's compactness theorem and the implicit function theorem.
  Since we have already proved that 
  \begin{eqnarray*}
&&b_k(\mathcal{Z}_{g_i}^{a-\delta})-b_{k-1}(\mathcal{Z}_{g_i}^{a-\delta})+ \cdots +(-1)^kb_0(\mathcal{Z}_{g_i}^{a-\delta})\\
&&\hspace{1cm} \leq c_{k,g_i}(a-\delta)-c_{k-1,g_i}(a-\delta)+ \cdots +(-1)^kc_{0,g_i}(a-\delta)
\end{eqnarray*}
for every $k\leq r$, we get 
$$
b_k(a)-b_{k-1}(a)+ \cdots +(-1)^kb_0(a)\leq c_k(a)-c_{k-1}(a)+ \cdots +(-1)^kc_0(a)
$$
for every $k\leq r$. Since $r\in \mathbb{Z}_+$ is arbitrary, we have finished the proof of Theorem \ref{theorem.closed.2}.

\end{proof}

\section{Boundary case and Compactness theory}

Let $(M^{n+1},g)$ be an $(n+1)$-dimensional compact, connected, Riemannian manifold with strictly convex boundary. We can suppose it is isometrically embedded in some Euclidean space $\mathbb{R}^L$. Let $\gamma^{n-1}\subset \partial M$ be an $(n-1)$-dimensional smooth, closed submanifold.

We denote by  $\mathcal{V}_n(M)$ the closure,  in the weak topology, of the space of $n$-dimensional rectifiable varifolds in $\mathbb{R}^L$ with support contained in $M$. Given $V\in \mathcal{V}_n(M)$, $||V||$ denotes the Radon measure in $M$ associated with $V$.  

We denote by  $A(p,s,r)$ the Euclidean annulus $\{x\in \R^L: s<|x-p|<r\}$.


The varifold $V$ is said to be $\gamma$-stationary if 
$$
\delta V(X)=0
$$
for every vector field $X$ that is tangential to $M$ and such that $X=0$ on $\gamma$.

Let
\begin{eqnarray*}
\mathfrak{M}(\gamma) = \{ \Sigma^n \subset M : \Sigma \text{ is a compact, embedded, smooth minimal } \\
\text{hypersurface, with } \partial \Sigma = \gamma\}.
\end{eqnarray*}
If we want to emphasize the dependence on the Riemannian metric $g$, we will write $\mathfrak{M}_g(\gamma)$ instead of $\mathfrak{M}(\gamma)$.

\subsection{Definition}
The Morse index of $\Sigma$ is defined as the number $\text{index}(\Sigma)$ of negative eigenvalues of the Jacobi operator $L_\Sigma$ of $\Sigma$, counted with multiplicities, acting on the space of smooth sections of the normal bundle of $\Sigma$ which vanish on $\partial \Sigma$.


We will extend to the case of minimal hypersurfaces with fixed boundary the compactness theory for closed minimal hypersurfaces developed by Sharp \cite{sharp}. 

Let
\begin{eqnarray*}
\mathfrak{M}_I(\gamma, \Lambda) =  \{ \Sigma \in  \mathfrak{M}(\gamma) : \Sigma \text{ connected}, \text{index}(\Sigma) \leq I, \text{ and } ||\Sigma||(M)\leq \Lambda\}
\end{eqnarray*}

The main result of this section is the following theorem.

\begin{thm}\label{strong-compactness1}
Consider $\{\Sigma_k\}_{k\geq 1} \subset \mathfrak{M}_I(\gamma, \Lambda)$. There exists a subsequence $\{\Sigma_k\}_{k}$ and $\Sigma \in \mathfrak{M}_I(\gamma, \Lambda)$ such that $\Sigma_k \rightarrow \Sigma$ as varifolds. The convergence is graphical and smooth with multiplicity one everywhere. If the $\Sigma_k$ are different from $\Sigma$, then the nullity of $\Sigma$ is at least one.

\end{thm}

\begin{proof}
We will start by proving the following claim:

\subsection{Claim}\label{claim1}
{\em There exists a subsequence $\{\Sigma_k\}_{k}$, an integral varifold $V$, and a finite subset $Y \subset M$ such that:
\begin{enumerate}
\item[(a)] $\Sigma_k \rightarrow V$ as varifolds,

\item[(b)] $V$ is $\gamma$-stationary,

\item[(c)] $Y\subset \gamma$ and $\#Y \leq I$,

\item[(d)] $\Sigma = supp(||V||)$ is connected, smooth in $M\setminus Y$,  and contains $\gamma$,

\item[(e)] $\Sigma_k\rightarrow \Sigma$ in the Hausdorff topology,

\item[(f)] $\Sigma_k \rightarrow \Sigma$ smooth and graphically, with multiplicity one on compact subsets of $M\setminus Y$, and

\item[(g)] $\Sigma\setminus Y$ is connected.
  
\end{enumerate} }

Let $\{\Sigma_k\}_{{k\geq 1}}$ be a sequence in $\mathfrak{M}_I(\gamma, \Lambda)$. Since $||\Sigma_k||(M)=\H^n(\Sigma_k)$ are bounded from above by $\Lambda$, we have that, up to a subsequence, $\Sigma_k$ converges in varifold sense to a varifold $V \in \mathcal{V}_n(M)$. Clearly, $V$ is $\gamma$-stationary. Moreover, minimality of  $\Sigma_k$ implies
\begin{equation}
||\delta (|\Sigma_k|) || \leq \H^{n-1}\llcorner \gamma +n ||A_M|| \H^{n}\llcorner \Sigma_k,
\end{equation}
where $|\Sigma_k|$ denotes the varifold induced by $\Sigma_k$ and $||A_M||$ denotes the $C^0$ norm of the second fundamental form of the inclusion $M \subset \R^L$. In particular, the first variation total measures $||\delta (|\Sigma_k|) ||(M)$ are also uniformly bounded. By Allard's Compactness Theorem for integral varifolds, (see 42.7 and 42.8  in \cite{simon}), $V$ is an  integral varifold.

Let $\Sigma = supp(||V||)$. By the Maximum Principle (see Proposition 7.1 of \cite{delellis-ramic}), $\Sigma \cap \partial M \subset \gamma$. We claim that $\Sigma$ contains $\gamma$, is connected, and $\Sigma_k\rightarrow \Sigma$ in the Hausdorff sense. Indeed, if $p \in \gamma$, the monotonicity formula at boundary points, obtained by Allard in \cite{allard-boundary}, implies a positive lower bound $\H^n(\Sigma_k\cap B_r(p))\geq Cr^n >0$ which is uniform on $k\in \N$ for small values of $r>0$. Then, by  continuity of the mass under varifold convergence, it follows that $p \in \Sigma$.  The Hausdorff convergence follows because if not there would be a sequence of points $p_{k_i} \in \Sigma_{k_i}$, with $k_i\rightarrow \infty$, at distance at least a fixed $d_1>0$ from $\Sigma$. Since $\gamma \subset \Sigma$, the interior monotonicity formula can be applied to give us a contradiction. Since $\Sigma_k$ are connected, $\Sigma$ is connected also.

Since $\text{index}(\Sigma_k) \leq I$ for every $k$, a standard argument shows that after maybe passing to a subsequence there exists a finite set $\tilde{Y}\subset M$, with $\# \tilde{Y} \leq I$, such that for every $x\in M\setminus \tilde{Y}$ there exists $r>0$ such that $\Sigma_k$ is stable  in $B_r(x)$ for every $k$. If $x\in M \setminus \partial M$, we can choose $r$ so that $B_r(x) \cap \partial M = \emptyset$. The Schoen-Simon theory \cite{schoen-simon} implies that $\Sigma$ is smooth and embedded  in $M \setminus (\partial M \cup \tilde{Y})$, perhaps with integer multiplicities,  and the convergence is locally graphical and smooth. Since $\partial M$ is strictly convex, the wedge property is satisfied near any $x \in \gamma$ (Lemma 9.1 of \cite{delellis-ramic}). Hence we can apply
 De Lellis and Ramic's Compactness Theorem (Theorem 7.4 of  \cite{delellis-ramic}) to obtain that for any $x \in \gamma \setminus \tilde{Y}$, there exists $r>0$ such that $\Sigma \cap B_r(x)$ is smooth and embedded with multiplicity one, $\partial (\Sigma \cap B_r(x)) = \gamma \cap B_r(x)$, with smooth and locally graphical convergence.  
 
 By Sharp \cite{sharp}, $\Sigma$ is smooth in $M\setminus \partial M$. In particular, for every $x \in \tilde{Y}\setminus \partial M$ only one connected component of $\Sigma\setminus \tilde{Y}$ has $x$ on its closure. 

 We claim now that the same holds even for $x\in \tilde{Y} \cap \gamma$. To see this, we start by observing that only one connected component of $\Sigma\setminus \tilde{Y}$ contains $\gamma\cap (B_r(x)\setminus \{x\})$, for some $r>0$. This is trivial when $4\leq (n+1)\leq 7$, since in this case $\gamma$ has dimension at least $2$, which implies that $\gamma\cap (B_r(x)\setminus \{x\})$ is connected. If $(n+1)=3$, choose a small $r>0$ such that $\partial B_r(x)$ is transversal to $\Sigma$. Then $\Sigma \cap \partial B_r(x)$ is a 1-dimensional compact submanifold  with exactly two boundary points, corresponding to $\gamma\cap \partial B_r(x)$. Therefore these points belong to the same component of $\Sigma\setminus \tilde{Y}$. We conclude our claim that no other component of $\Sigma\setminus \tilde{Y}$ has $x$ on its closure by contradiction; otherwise, such a component would be minimal in $B_r(x)\setminus \{x\}$ and touch $\partial M$ at $x$ only, which contradicts the maximum principle of White \cite{white-maximum-principle} since $\partial M$ is convex.

 In particular, $\Sigma\setminus \tilde{Y}$ is connected and hence has multiplicity one. By Allard's regularity theory, in the interior and  boundary cases (\cite{allard, allard-boundary}), the convergence $\Sigma_k\rightarrow V$ is smooth and graphical with a single sheet on compact subsets of $M\setminus Y$, where $Y=\tilde{Y} \cap \partial M$. We also know that $\Sigma$ is smooth in $M\setminus Y$ and in $M\setminus \partial M$. This proves Claim \ref{claim1}.
 
\subsection{Claim}\label{remove-bdry-sing}
{\em $\Sigma$ is smooth, embedded  and $ind(\Sigma)\leq I$. 
The convergence $\Sigma_k \rightarrow \Sigma$ is graphical and smooth with one sheet on $M$.  If $\Sigma_k\neq \Sigma$ for sufficiently large $k$, the nullity of $\Sigma$ is at least one.}
\medskip

We follow the notation from Claim \ref{claim1}. In order to prove regularity at $y\in Y$, it suffices to show that any tangent cone of $V$ at $y$   is an $n$-dimensional half-space with multiplicity one. The result would follow from Allard's boundary regularity theorem in \cite{allard-boundary}. We start by observing that Claim \ref{claim1} and Proposition 6.3 of \cite{montezuma} imply that any such $C$ can be represented as a sum $C = \sum_{i=1}^l c_i \pi_i$, where $\{\pi_i\}_{i=1}^l$ and $\{c_i\}_{i=1}^l$ are collections of $n$-dimensional half-spaces that contain $T_y\gamma$ and positive integers, respectively.

Observe now that the argument of Claim 2 on page 326 of \cite{sharp} can also be applied at boundary points. In particular, there exists $\varepsilon>0$ such that $\Sigma$ is stable in $B_{\varepsilon}(y)\setminus \{y\}$, for every $y \in Y$, with respect to vector fields satisfying $X|_{\gamma}=0$. This is possible because we still have that $\delta^2 \Sigma_k(X)\rightarrow \delta^2 V (X)$, for compactly supported $C^1$ vector fields $X$ that vanish along $\gamma$.

Let $r_1 > r_2 > \cdots >0$ be such that $r_j\rightarrow 0$, and $(\eta_{y, r_j})_{\#}V \rightarrow C$, as $j\rightarrow \infty$. It follows from the Compactness Theorem 8.4 for stable minimal hypersurfaces of De Lellis and Ramic \cite{delellis-ramic}, slightly modified to allow for varying metrics and boundaries, that a subsequence of $(\eta_{y, r_j})_{\#}V \llcorner A(O,1,2)$ converges graphically and smoothly to $W$ with $\partial W=T_y\gamma$. Since each component of $W$ that intersects $T_y(\partial M)$ must have multiplicity one, and $W=C=\sum_{i=1}^l c_i \pi_i$, we conclude that $C$ is a single half-space with multiplicity one.  This proves the regularity of $\Sigma$ at points $y \in Y$.

By Allard's regularity theorem for boundary points (\cite{allard-boundary}), we obtain that $\Sigma_k \rightarrow \Sigma$ graphically and smoothly with one sheet on $M$. In particular, ${\rm index}(\Sigma) \leq I$. If $\Sigma_k\neq \Sigma$ for sufficiently large $k$, and since the convergence is smooth with multiplicity one, a standard argument gives a nontrivial Jacobi field. This finishes the proof of Claim \ref{remove-bdry-sing}.

 \end{proof}
 
 \subsection{Remark}
 The compactness statement of Theorem \ref{strong-compactness1}, with the exception of the nullity conclusion, holds true in the case of varying and converging Riemannian metrics $g_k\rightarrow g$ and boundaries
 $\gamma_k\rightarrow \gamma$. The proof follows by simple modifications (see Remark 7.4 of \cite{delellis-ramic}).
 
 \section{Morse inequalities in the boundary case}
 
 Let $3\leq (n+1)\leq 7$. We suppose that $\partial M$ is strictly convex and that $M$ contains no closed, embedded, smooth minimal hypersurfaces. Let $\gamma^{n-1}\subset \partial M$ be an $(n-1)$-dimensional smooth, closed submanifold.
By Theorem 2.1 of White \cite{white-isoperimetric}, there exists a constant $c>0$ such that
 $$
 \text{area}_n(\Sigma) \leq c \left(\text{area}_{n-1}(\partial \Sigma)+ \int_\Sigma |H_\Sigma| \, d\Sigma\right)
 $$
 for any compact hypersurface $\Sigma \subset M$.
  In particular, there exists a constant $\Lambda>0$ such that
 $$
 ||\Sigma||(M)\leq \Lambda
 $$
 for any compact minimal hypersurface $\Sigma \subset M$ with $\partial \Sigma=\gamma$. 
 
 Let us assume that every compact minimal hypersurface $\Sigma \subset M$ with $\partial \Sigma=\gamma$ is nondegenerate, meaning that it admits no nontrivial
 Jacobi fields vanishing on $\gamma$. This condition holds for Baire generic boundaries $\gamma \subset \partial M$ (\cite{white-parametric}, and Section 7.1 of \cite{AmbCarSha}).
 
 Given $k\in \mathbb{Z}_+$, let $c_k(\gamma)$ denote the number of embedded, compact minimal hypersurfaces $\Sigma \subset M$ with $\partial \Sigma=\gamma$ and ${\rm index}(\Sigma)=k$. If we want to emphasize the dependence on the Riemannian metric $g$, we will write $c_{k,g}(\gamma)$ instead of $c_k(\gamma)$.

 \subsection{Proposition}
 {\em $c_k(\gamma)<\infty$ for every $k$.}
 
 \begin{proof}
 Suppose $c_k(\gamma)= \infty$ for some $k$. Let $\{\Sigma_j\}_j$ be an infinite sequence of compact minimal hypersurfaces $\Sigma_j \subset M$ with $\partial \Sigma_j=\gamma$, ${\rm index}(\Sigma_j)=k$ and $\Sigma_j \neq \Sigma_{j'}$ whenever $j \neq j'$. 
 
 We have $||\Sigma_j||(M)\leq \Lambda$ for every $j$. By the Monotonicity Formula, the number of connected components is uniformly bounded. By passing to a subsequence, we can suppose the number of connected components of $\Sigma_j$ is equal to a fixed number $p$ for every $j$. There are no closed components by assumption. Let $\Sigma_j^{(1)}, \dots, \Sigma_j^{(p)}$ be the components of $\Sigma_j$. By Theorem \ref{strong-compactness1}, again passing to a subsequence, we can suppose $\Sigma_j^{(l)}$ converges smooth and graphically with multiplicity one to some compact minimal hypersurface $\tilde{\Sigma}^{(l)}$ with $\partial \tilde{\Sigma}^{(l)}\subset \partial M$, for every $l=1, \dots, p$. Since $\Sigma_j$ is embedded for every $j$, the Maximum Principle implies that the collection $\{\tilde{\Sigma}^{(1)}, \dots, \tilde{\Sigma}^{(p)}\}$ is disjoint and $\partial (\tilde{\Sigma}^{(1)} \cup \dots \cup \tilde{\Sigma}^{(p)})=\gamma$. Since $\Sigma_j \neq \Sigma_{j'}$ whenever $j \neq j'$, Theorem \ref{strong-compactness1} implies that the nullity of $\tilde{\Sigma}^{(1)} \cup \dots \cup \tilde{\Sigma}^{(p)}$ is at least one. This contradicts the choice of $\gamma$, which finishes the proof of the proposition.

 \end{proof}
 
 Let $\mathcal{Z}_n(M,\gamma;\mathbb{Z}_2)$ denote the space of $n$-dimensional modulo  2 flat chains $T$ with $\partial T=\gamma$. Let ${\bf I}_{n+1}(M;\mathbb{Z}_2)$ denote the space of $(n+1)$-dimensional modulo 2 flat chains $U$ with $\text{support}(U)\subset M$. These spaces are endowed with the flat topology. When endowed with the ${\bf F}$ metric they will be denoted by $\mathcal{Z}_n(M,\gamma;{\bf F}; \mathbb{Z}_2)$ and ${\bf I}_{n+1}(M;{\bf F}; \mathbb{Z}_2)$, and with the mass metric by $\mathcal{Z}_n(M,\gamma;{\bf M}; \mathbb{Z}_2)$ and ${\bf I}_{n+1}(M;{\bf M}; \mathbb{Z}_2)$.

 \subsection{Proposition}\label{contractibility} {\em The topological spaces $\mathcal{Z}_n(M,\gamma;\mathbb{Z}_2)$ and ${\bf I}_{n+1}(M;\mathbb{Z}_2)$ are  contractible.}
 
 \begin{proof}
 The contractibility of ${\bf I}_{n+1}(M;\mathbb{Z}_2)$ follows exactly as in Claim 5.3 of \cite{marques-neves-lower-bound}. We define $H:[0,1] \times {\bf I}_{n+1}(M;\mathbb{Z}_2) \rightarrow {\bf I}_{n+1}(M;\mathbb{Z}_2)$ by
$$
H(t,U) = U \llcorner \{f\leq t\},
$$
where $f:M \rightarrow [0,1]$ is a Morse function on the manifold with boundary $M$.
The map $H$ is continuous, $H(1,U)=U$ and $H(0,U)=0$ for every $U \in {\bf I}_{n+1}(M;\mathbb{Z}_2)$. 
 
 Now note that $H_n(M^{n+1}, \mathbb{Z}_2)=0$. This has to be true because otherwise one could minimize area inside a nontrivial homology class and produce a closed minimal hypersurface inside $M$. Fix a reference $T \in \mathcal{Z}_n(M,\gamma;\mathbb{Z}_2)$. For any $T'\in \mathcal{Z}_n(M,\gamma;\mathbb{Z}_2)$, we have $\partial(T'-T)=0$. Since $H_n(M^{n+1}, \mathbb{Z}_2)=0$, we can find $U \in {\bf I}_{n+1}(M;\mathbb{Z}_2)$ such that $T'-T=\partial U$. We claim that $U$ is unique. In fact, if $T'-T=\partial V$, $V\in  {\bf I}_{n+1}(M;\mathbb{Z}_2)$, then $W=U-V$ is a top-dimensional chain with $\partial W=0$. Because $\text{support}(W)\subset M$ and $M$ is connected with  nontrivial boundary, the Constancy Theorem for mod 2 flat chains implies $W=0$.  Hence $U$ is unique. In particular, the continuous map
\begin{eqnarray*}
  {\bf I}_{n+1}(M;\mathbb{Z}_2) &\rightarrow& \mathcal{Z}_n(M,\gamma;\mathbb{Z}_2) \\
  U &\mapsto& T+\partial U
\end{eqnarray*}
is a bijection. The Federer-Fleming Isoperimetric Inequality (\cite{federer-fleming}) implies that the inverse of the above map is also continuous, hence the map is a homeomorphism. The Proposition follows.

 \end{proof}
 
 As a consequence, if $b_k(\gamma)$ denotes the $k$-th Betti number of $\mathcal{Z}_n(M,\gamma;\mathbb{Z}_2)$, then $b_0(\gamma)=1$ and $b_k(\gamma)=0$ for every $k\geq 1$. The next Proposition shows that $\mathcal{Z}_n(M,\gamma;{\bf F}; \mathbb{Z}_2)$ has the same Betti numbers
 as $\mathcal{Z}_n(M,\gamma;\mathbb{Z}_2)$.
 
 \subsection{Proposition}\label{betti.F} {\em We have:
 \begin{eqnarray*}
 b_0\left(\mathcal{Z}_n(M,\gamma;{\bf F}; \mathbb{Z}_2)\right) &=& 1,\\
  b_k\left(\mathcal{Z}_n(M,\gamma;{\bf F}; \mathbb{Z}_2)\right) &=& 0  \text{\, \, \ for\, \, } k\geq 1.
 \end{eqnarray*}
 }
 
 \begin{proof}
Fix a reference $T \in \mathcal{Z}_n(M,\gamma;{\bf F};\mathbb{Z}_2)$. Given any $T'\in \mathcal{Z}_n(M,\gamma;{\bf F};\mathbb{Z}_2)$, the proof of Proposition \ref{contractibility} gives a map $\Phi:[0,1] \rightarrow \mathcal{Z}_n(M,\gamma;\mathbb{Z}_2)$ that is continuous in the flat topology, has no concentration of mass:
$$
\lim_{r\rightarrow 0} \sup_{x \in [0,1]} \{{\bf M}(\Phi(x) \llcorner B(p,r)): p\in M\} =0,
$$
with $\Phi(0)=T$ and $\Phi(1)=T'$. The interpolation results of Section 3 of \cite{montezuma} produce out of $\Phi$ a continuous map $\Psi:[0,1] \rightarrow \mathcal{Z}_n(M,\gamma;{\bf M}; \mathbb{Z}_2)$ with $\Psi(0)=T$ and $\Psi(1)=T'$. Since $\Psi$ is also continuous in the ${\bf F}$-metric, we get that
$$
b_0\left(\mathcal{Z}_n(M,\gamma;{\bf F}; \mathbb{Z}_2)\right) = 1.
$$

This can be generalized in the following way. Let $X$ be a finite dimensional compact simplicial complex, and let $\Phi: X\rightarrow \mathcal{Z}_n(M,\gamma;{\bf F};\mathbb{Z}_2)$ be a continuous map in the ${\bf F}$-metric. This implies continuity in the flat topology and no concentration of mass:
$$
\lim_{r\rightarrow 0} \sup_{x \in X} \{{\bf M}(\Phi(x) \llcorner B(p,r)): p\in M\} =0.
$$
According to the proof of Proposition \ref{contractibility}, we can write $\Phi(x) = T+\partial U(x)$ where $U:X \rightarrow {\bf I}_{n+1}(M,\mathbb{Z}_2)$ is continuous. Note that $x \mapsto \partial U(x)$ does not concentrate mass. Hence $\Psi:X\times [0,1] \rightarrow \mathcal{Z}_n(M,\gamma;\mathbb{Z}_2)$ defined by
$$
\Psi(x,t) = T+\partial \left(U(x) \llcorner \{f\leq t\}\right)
$$
is continuous in the flat topology with no concentration of mass:
$$
\lim_{r\rightarrow 0} \sup_{x \in X, t\in [0,1]} \{{\bf M}(\Psi(x,t) \llcorner B(p,r)): p\in M\} =0,
$$
and $\Psi(x,0)=T$, $\Psi(x,1) =\Phi(x)$ for every $(x,t)\in X\times [0,1]$. Section 3 of \cite{montezuma} gives a sequence  of continuous maps
$\Psi_i: X\times [0,1] \rightarrow \mathcal{Z}_n(M,\gamma;{\bf M};\mathbb{Z}_2)$ such that $\Psi_i(x,0)=T$ for every $x \in X$ and
$$
\lim_{i \rightarrow \infty} \sup_{x\in X} {\bf F}(\Psi_i(x,1), \Psi(x,1))=0.
$$
The interpolation work of Section 3 of \cite{marques-neves-lower-bound} can be extended,  using the techniques of \cite{montezuma},  to the setting of constrained boundary . In particular, as in Theorem 3.8 of \cite{marques-neves-lower-bound}, for sufficiently large $i$ we can find a homotopy continuous in the ${\bf F}$-metric between $x \mapsto  \Psi_i(x,1)$ and $x \mapsto \Psi(x,1)$. This implies $\Phi: X\rightarrow \mathcal{Z}_n(M,\gamma;{\bf F};\mathbb{Z}_2)$ is homotopic in the ${\bf F}$-metric to a constant map.  Since $\Phi$ is arbitrary, we conclude that
$$
b_k\left(\mathcal{Z}_n(M,\gamma;{\bf F}; \mathbb{Z}_2)\right) = 0  \text{\, \, \ for\, \, } k\geq 1.
$$
This finishes the proof of the Proposition.

 \end{proof}
 
 \subsection{Theorem}\label{morse.inequalities}{\em The Strong Morse Inequalities for the area functional hold, i.e.
 $$
 c_k(\gamma)-c_{k-1}(\gamma)+ \cdots +(-1)^kc_0(\gamma) \geq  b_k(\gamma)-b_{k-1}(\gamma)+ \cdots +(-1)^kb_0(\gamma) = (-1)^k
 $$
 for every $k\geq 0$.}
 
Before proving Theorem \ref{morse.inequalities}, we will prove the following auxiliary result:

\subsection{Proposition}\label{rational.independence} {\em There exists a sequence of Riemannian metrics $(g_i)_{i\in \mathbb{N}}$ converging to $g$ in the smooth topology so that for each $i\in \mathbb{N}$:
\begin{itemize}
\item $\partial M$ is strictly convex in $(M,g_i)$ and $(M,g_i)$ contains no closed, embedded, minimal hypersurface;
\item every $g_i$-minimal hypersurface with boundary $\gamma$ is $g_i$-nondegenerate;
\item and if 
$$
p_1 \cdot {\rm area}_{g_i}(\Sigma_1) + \cdots + p_N \cdot {\rm area}_{g_i}(\Sigma_N)=0,
$$
with $\{p_1, \dots, p_N\} \subset \mathbb{Z}$, $\{\Sigma_1, \dots, \Sigma_N\}\subset \mathfrak{M}_{g_i}(\gamma),$ and $\Sigma_k \neq \Sigma_l$ whenever $k\neq l$, then
$$
p_1 = \cdots = p_N = 0.
$$
\end{itemize}
 }

\begin{proof}
Let $\hat{\mathcal{M}}$ be the set of smooth Riemannian metrics on $M$ such that $\partial M$ is strictly convex and $M$ contains no closed, embedded, minimal hypersurface. We claim that $\hat{\mathcal{M}}$ is open in the space of all metrics $\mathcal{M}$ on $M$, with respect to the $C^\infty$ topology. 
The convexity condition is clearly open. Now, if the second condition is not satisfied then one can minimize the boundary area over  all regions that enclose a certain closed minimal hypersurface $\Sigma$ (see Remark 2.6 of \cite{white-isoperimetric}). This  implies there is a stable, closed, embedded 
minimal hypersurface inside $M$ with area less than or equal to the area of the boundary. Sharp's Compactness Theorem (\cite{sharp}) allows one to take a limit of these stable hypersurfaces. This shows $\hat{\mathcal{M}}$ is open. In particular, $\hat{\mathcal{M}}$ is a Baire space.

Let $\mathcal{U}_p$ be the set of Riemannian metrics $g\in \hat{\mathcal{M}}$ such that:
\begin{itemize}
\item every $g$-minimal hypersurface with boundary $\gamma$ and index at most $p$ is $g$-nondegenerate;
\item and if 
$$
m_1 \cdot {\rm area}_{g}(\Sigma_1) + \cdots + m_N \cdot {\rm area}_{g}(\Sigma_N)=0,
$$
$\{m_1, \dots, m_N\} \subset \mathbb{Z}$, $\{\Sigma_1, \dots, \Sigma_N\}\subset \mathfrak{M}_{g}(\gamma),$ $|m_k|\leq p$, $\text{index}(\Sigma_k) \leq p$ for every $k$, and $\Sigma_k \neq \Sigma_l$ whenever $k\neq l$, then
$$
m_1 = \cdots = m_N = 0.
$$
\end{itemize}

\subsection{Claim} {\em The set $\mathcal{U}_p$ is open and dense in $\hat{\mathcal{M}}$, for every $p\in \mathbb{Z}_+$.}

\medskip

If $g \in \hat{\mathcal{M}}$, then Theorem 2.1 of White \cite{white-isoperimetric} implies there exists a constant $c>0$ such that $||V||\leq c ||\delta V||$ for every $n$-dimensional varifold in $M$. Now this gives that for any metric sufficiently close to $g$, we will have the inequality $||V||\leq (c+\varepsilon) ||\delta V||$ for some $\varepsilon>0$. (Notice this yields another proof of the openness of $\hat{\mathcal{M}}$.) Hence in that neighborhood there is a uniform bound for the area of any minimal hypersurface with boundary $\gamma$.
We can apply Theorem \ref{strong-compactness1} and the Implicit Function Theorem to conclude the set $\mathcal{U}_p$ is open.

Let $g \in \hat{\mathcal{M}}$ and let $\mathcal{V}\subset  \hat{\mathcal{M}}$ be a $C^\infty$-neighborhood of $g$. There exists a submanifold $\gamma'$  arbitrarily close to $\gamma$ such that any $g$-minimal hypersurface with boundary $\gamma'$ is $g$-nondegenerate. By pulling the metric $g$ back by a diffeomorphism of $M$ arbitrarily close to the identity, sending $\gamma$ to $\gamma'$, we find a metric
$g'\in \mathcal{V}$ such that every $g'$-minimal hypersurface with boundary $\gamma$ is $g'$-nondegenerate.

By Theorem \ref{strong-compactness1}, there are only finitely many $g'$-minimal hypersurfaces with boundary $\gamma$ and index at most $p$. Let $\{S_1, \dots, S_q\}$ be the collection of such hypersurfaces, with $S_k\neq S_l$ whenever $k\neq l$.

Recall that if $\tilde{g}=\exp(2\phi)g'$, then the second fundamental form of $\Sigma$ with respect to $\tilde{g}$ is given by (Besse \cite{besse}, Section 1.163)
\begin{eqnarray*}\label{second.fundamental.form}
A_{\Sigma, \tilde{g}} =  A_{\Sigma,g'} -  g' \cdot (\nabla \phi)^\perp,
\end{eqnarray*}
where $(\nabla \phi)^\perp(x)$ is the component of $\nabla \phi$ normal to $T_x\Sigma$.

We can pick $p_l \in \text{int}(S_l) \setminus (\cup_{k\neq l} S_k)$ for every $l=1, \dots, q$ (see the proof of Lemma 4 of 
\cite{marques-neves-song}). Let $\varepsilon>0$ be sufficiently small so that $B_\varepsilon(p_k)\cap B_\varepsilon(p_l) = \emptyset$ whenever $k\neq l$ and $B_\varepsilon(p_l) \cap \partial M= \emptyset$, $B_\varepsilon(p_l) \cap (\cup_{k\neq l} S_k) = \emptyset$ for every $l=1, \dots, q$. We choose a nonnegative function $f_l\in C_c^\infty(B_\varepsilon(p_l))$, ${f_l}_{|S_l}\not \equiv 0$, such that $(\nabla_{g'} f_l)(x) \in T_xS_l$ for every $x\in S_l$. Hence $S_l$ is still minimal with respect to the metric $\hat{g}(t_1,\dots,t_q) = \exp(2(t_1f_1+\cdots+t_qf_q))g'$, for every $l=1, \dots, q$. 

Let $(t_1^{(i)}, \dots, t_q^{(i)})\in (0,1]^q$ be a sequence converging to zero so that, by putting $g_i=\hat{g}(t_1^{(i)}, \dots, t_q^{(i)})$, we have that  the real numbers
$$
{\rm area}_{g_i}(S_1), \dots, {\rm area}_{g_i}(S_q)
$$
are linearly independent over $\mathbb{Q}$. Compactness Theorem \ref{strong-compactness1}, together with the fact that $S_l$ is $g'$-nondegenerate for every $l=1, \dots, q$, implies that for sufficiently large $i$ the only $g_i$-minimal hypersurfaces with boundary $\gamma$ and index at most $p$ are $S_1, \dots, S_q$. Hence
$g_i \in \mathcal{V} \cap \mathcal{U}_p$ for sufficiently large $i$. This proves density of $\mathcal{U}_p$, which finishes the proof of the claim.

\medskip

The claim implies the set $X=\cap_p \, \mathcal{U}_p$ is Baire-residual in $\hat{\mathcal{M}}$, hence it is also dense. This finishes the proof of the Proposition.

\end{proof}

\begin{proof}[Proof of Theorem \ref{morse.inequalities}]

Let $g_i$ converging to $g$ be like in Proposition \ref{rational.independence}. Theorem \ref{strong-compactness1} and the Implicit Function Theorem imply that for sufficiently large $i$ we have $c_{j,g}(\gamma)=c_{j,g_i}(\gamma)$ for every $j\leq k$. This means that, without loss of generality, we can suppose that if 
$$
p_1 \cdot {\rm area}_{g}(\Sigma_1) + \cdots + p_N \cdot {\rm area}_{g}(\Sigma_N)=0,
$$
with $\{p_1, \dots, p_N\} \subset \mathbb{Z}$, $\{\Sigma_1, \dots, \Sigma_N\}\subset \mathfrak{M}_{g}(\gamma),$ and $\Sigma_m \neq \Sigma_l$ whenever $m\neq l$, then
$$
p_1 = \cdots = p_N = 0.
$$
In particular, for each $t \in \mathbb{R}$ there can be at most one element $\Sigma \in \mathfrak{M}(\gamma)$ with $\text{area}(\Sigma)=t$.

Let $k\in \mathbb{N}$. We denote by $c_j(t)=c_j(\gamma;t)$ the number of elements of $\mathfrak{M}(\gamma)$ with index equal to $j$ and area less than or equal to $t$. Recall that $\text{area}(\Sigma) \leq \Lambda$ for every $\Sigma \in \mathfrak{M}(\gamma)$.

Let 
$$
\mathcal{Z}^t=\{T \in \mathcal{Z}_n(M,\gamma;{\bf F}; \mathbb{Z}_2): {\bf M}(T) < t\}.
$$

 All the homology groups will be computed with coefficients in $\mathbb{Z}_2$. A $k$-dimensional homology class is represented by finite chains of singular $k$-simplices $\sum s_i$, where $s_i:\Delta^k \rightarrow \mathcal{Z}_n(M,\gamma;{\bf F};\mathbb{Z}_2)$ is a continuous map defined on the standard $k$-simplex $\Delta^k$ for every $i$.
 
 \subsection{Homology Min-Max Theorem}{\em Let $\sigma \in H_k(\mathcal{Z}^t, \mathcal{Z}^s)$ be a nontrivial homology class, and let
 $$
 W(\sigma) = \inf_{[\sum s_i]=\sigma}\sup_{i, x\in \Delta^k} {\bf M}(s_i(x)).
 $$
 Then $W(\sigma)\geq s$, and there exists $\Sigma \in \mathfrak{M}(\gamma)$ with ${\rm index}(\Sigma)=k$ and
 $$
 {\rm area}(\Sigma) = W(\sigma).
 $$
 }
 \begin{proof}
 If $W(\sigma)<s$, then there is a representative $[\sum s_i]=\sigma$ with $ s_i(\Delta^k)\subset \mathcal{Z}^s$ for every $i$. This implies $\sigma=0$, which is
 a contradiction. Hence $W(\sigma)\geq s$.
 
 Let $\{\sum s_i^{(j)}\}_j$ be a sequence of representatives ($[\sum s_i^{(j)}]=\sigma$)  such that 
 $$
\lim_{j\rightarrow \infty} \sup_{i, x\in \Delta^k} {\bf M}(s_i^{(j)}(x))=W(\sigma).
 $$
 Associated to  the  chain $\sum s_i^{(j)}$ we have a $\Delta$-complex $X^{(j)}$ and a map $\Phi^{(j)}: X^{(j)} \rightarrow \cup_i s_i^{(j)}(\Delta^k)$ (see Section 2.1 of \cite{hatcher}) that is continuous in the ${\bf F}$-metric. The boundary $\partial X^{(j)}$ is the union of $(k-1)$-faces of  $\sum s_i^{(j)}$ that do not cancel out
 in the calculation of $\partial (\sum s_i^{(j)})$. Because the  chain $\sum s_i^{(j)}$ is a relative cycle, we have that $\Phi^{(j)}(\partial X^{(j)})\subset \mathcal{Z}^s$.
 
 Let $\Pi^{(j)}$ be the homotopy class of $\Phi^{(j)}$ relative to $\partial X^{(j)}$. If
 $$
 {\bf L}(\Pi^{(j)}) =\inf_{\Phi \in \Pi^{(j)}}\sup_{x\in X^{(j)}} {\bf M}(\Phi(x)),
 $$
 then
 $$
 W(\sigma) \leq  {\bf L}(\Pi^{(j)})\leq \sup_{i, x\in \Delta^k} {\bf M}(s_i^{(j)}(x)).
 $$
 In particular, $\lim_{j\rightarrow \infty} {\bf L}(\Pi^{(j)})=W(\sigma)$.
 
 Notice that in Section 3 of \cite{marques-neves-index}, a homotopy class is defined in an unusual way: the homotopy class of an ${\bf F}$-continuous map
 is defined as the class  of all ${\bf F}$-continuous maps that are homotopic to the original one in the flat topology.
But Proposition 3.5 of \cite{marques-neves-infinitely}  has been upgraded to the mass topology in Proposition 3.2 of \cite{marques-neves-lower-bound}. An inspection of Section 3 of \cite{marques-neves-index} shows that the min-max theory also works for usual ${\bf F}$-continuous homotopy classes. The adaptations
to the boundary case made in \cite{montezuma} lead to the existence of $\Sigma^{(j)} \in \mathfrak{M}(\gamma)$ with
${\rm area}(\Sigma_j)={\bf L}(\Pi_j)$. Deformation Theorem A of \cite{marques-neves-index} can also be adapted to the boundary case, by requiring that the diffeomorphisms in the definition of $k$-unstable varifolds (Definition 4.1 of \cite{marques-neves-index}) be such that they fix the boundary $\partial M$. Theorem 6.1 of \cite{marques-neves-index} adapted to the boundary case allows us to choose $\Sigma_j$ such that ${\rm index}(\Sigma_j) \leq k$. By Compactness Theorem \ref{strong-compactness1},  there must be some $j_0$ such that $\Sigma_j=\Sigma_{j_0}$ for all $j\geq j_0$. This means that for $j\geq j_0$, we have
$$
{\bf L}(\Pi^{(j)})=W(\sigma).
$$
Since Theorem 5 of White (\cite{white-minmax}) holds also in the boundary case, the techniques of \cite{marques-neves-lower-bound} can be adapted to give ${\rm index}(\Sigma_j)=k$. This finishes the proof of the Homology Min-Max Theorem.

 \end{proof}

 \subsection{Claim} {\em We have
\begin{eqnarray*}
b_0(\mathcal{Z}^{\Lambda +1}) &=& 1, \\
 b_k(\mathcal{Z}^{\Lambda +1}) &=& 0 \, \, \, \text{for}\, \, \, k\geq 1.
\end{eqnarray*}
 }
 
 From the long exact homology sequence of pairs, it is enough to show
 $$
 b_r(\mathcal{Z}(M,\gamma;{\bf F};\mathbb{Z}_2),\mathcal{Z}^{\Lambda +1})=0
 $$
 for every $r\geq 0$. But this follows immediately from the Homology Min-Max Theorem. The proof of the strong Morse inequalities
 proceed exactly as in the closed case.
 


\end{proof}

\bibliographystyle{amsbook}

\end{document}